\documentclass[a4paper,11pt,reqno]{amsart}

\usepackage{amsmath}
\usepackage{amsthm}
\usepackage{amssymb}
\usepackage{latexsym}
\usepackage[dvipdfmx]{graphicx,color}
\usepackage{booktabs}

\numberwithin{equation}{section}

\newtheorem{thm}{Theorem}[section]
\newtheorem{prop}[thm]{Proposition}
\newtheorem{lem}[thm]{Lemma}

\theoremstyle{definition}
\newtheorem{defn}[thm]{Definition}

\theoremstyle{remark}
\newtheorem{rem}[thm]{Remark}

\allowdisplaybreaks[3]

\newcommand{\de}{\delta}

\newcommand{\e}{\varepsilon}

\newcommand{\sgm}{\sigma}
\renewcommand{\th}{\theta}
\newcommand{\om}{\omega}

\newcommand{\p}{\partial}
\newcommand{\I}{\infty}
\newcommand{\Sc}[1]{\mathcal{#1}}
\newcommand{\F}{\Sc{F}}

\newcommand{\Bo}[1]{\mathbb{#1}}
\newcommand{\R}{\Bo{R}}
\newcommand{\T}{\Bo{T}}

\newcommand{\lec}{\lesssim}
\newcommand{\gec}{\gtrsim}

\newcommand{\bbar}{\overline}
\newcommand{\ti}{\widetilde}

\newcommand{\shugo}[1]{\{ #1\}}
\newcommand{\Shugo}[2]{\big\{ \, #1 \, \big| \, #2 \, \big\}}

\newcommand{\LR}[1]{{\langle #1 \rangle }}
\newcommand{\chf}[1]{\mathbf{1}_{#1}}

\newcommand{\norm}[2]{\| #1 \| _{#2}}
\newcommand{\tnorm}[2]{\| #1 \| _{#2}}

\newcommand{\eq}[2]{\begin{equation} \label{#1} \begin{split} #2 \end{split} \end{equation}}
\newcommand{\eqq}[1]{\begin{align*} #1 \end{align*}}
\newcommand{\eqqed}[1]{\begin{equation} \begin{split} #1 \qedhere \end{split} \end{equation}}
\newcommand{\eqs}[1]{\begin{gather*} #1 \end{gather*}}
\newcommand{\mat}[1]{\begin{smallmatrix} #1 \end{smallmatrix}}

\newcommand{\hx}{\hspace{10pt}}

\renewcommand{\H}{\mathcal{H}}

\title[Unconditional uniqueness for periodic mBO]{Unconditional uniqueness for the periodic modified Benjamin-Ono equation by normal form approach}
\author[N. Kishimoto]{Nobu Kishimoto}
\address{Research Institute for Mathematical Sciences, Kyoto University, Kyoto 606-8502, Japan}
\email{nobu@kurims.kyoto-u.ac.jp}

\begin{document}

\begin{abstract}
We show that the solution (in the sense of distribution) to the Cauchy problem with the periodic boundary condition associated with the modified Benjamin-Ono equation
is unique in $L^\I _t(H^s(\T ))$ for $s>1/2$.
The proof is based on the analysis of a normal form equation obtained by infinitely many reduction steps using integration by parts in time after a suitable gauge transform.
\end{abstract}

\maketitle

\section{Introduction}

\smallskip
In this article, we study uniqueness property of solutions to the Cauchy problem associated with the modified Benjamin-Ono equation with the periodic boundary condition:
\eq{mBO}{\p _tu=-\H \p _x^2u+\sgm u^2\p _xu,\qquad (t,x)\in (0,T)\times \T ,}
with initial datum given in Sobolev spaces,
\eq{mBO-id}{u(0,x)=u_0(x)\in H^s(\T ),\qquad x\in \T ,}
where $\T :=\R /2\pi \Bo{Z}$.
The unknown function $u(t,x)$ and the initial datum $u_0(x)$ are assumed to be real valued, and $\sgm \in \shugo{\pm 1}$, with the sign $+1$ (\mbox{resp.} $-1$) corresponding to the focusing (\mbox{resp.} defocusing) case.
$\H$ denotes the periodic Hilbert transform defined by the Fourier multiplier with symbol $-i\,\mathrm{sgn}(n)\chf{n\neq 0}$.

Well-posedness of the Cauchy problem \eqref{mBO}--\eqref{mBO-id} has been extensively studied in both non-periodic and periodic settings; see \cite{I86,KPV94,KK03,MR04a,MR04b,KT06,MR09,G11,GLM14}.
The equation \eqref{mBO} has three conserved quantities (formally):
\eqq{\int _{\T}u(t,x)\,dx,\qquad \int _{\T}u(t,x)^2\,dx,\qquad \int _{\T}\Big\{ \frac{1}{2}u(t,x)\H \p _xu(t,x)-\frac{\sgm}{12}u(t,x)^4\Big\} \,dx.}
Hence, under the assumption that solutions are real-valued, suitable local well-posedness in the energy space $H^{1/2}$ would extend to global result by these conservation laws in the defocusing case and also in the case of focusing sign and small-in-$L^2$ initial data. 
This was achieved by Kenig and Takaoka~\cite{KT06} in the non-periodic case and by Guo et al.~\cite{GLM14} in the periodic case.
Both of them relied crucially on the technique of gauge transform, which was first introduced by Tao~\cite{T04} in the Benjamin-Ono context.
See \cite{G11,Sp} for an approach without performing gauge transform.
The regularity $s=1/2$ was shown to be sharp in the non-periodic case in the sense that the solution map is not locally uniformly continuous in $H^s$ for $s<1/2$ (\cite{KT06}), while local-in-time a priori $H^s$-bound of smooth solutions was obtained for $s>1/4$ (\cite{G11,Sp}).

On the other hand, well-posedness results in \cite{KT06,GLM14} used auxiliary spaces such as $L^p_xL^q_t$-type spaces and Bourgain-type spaces.
Therefore, uniqueness of the solution to the Cauchy problem in such a class as $C([0,T];H^s)$ without intersecting any auxiliary space, which is called \emph{unconditional uniqueness}, can be asked as a natural question.
Unconditional uniqueness of solutions in $H^s$ is obtained for $s>3/2$ both on $\R$ and on $\T$ from the proof of well-posedness in \cite{I86} based on the energy method.
However, to our knowledge, there is no result on unconditional uniqueness for the modified Benjamin-Ono equation below $H^{3/2}$.

In this article, we prove the following unconditional uniqueness result in the periodic case.
\begin{thm}\label{thm}
The solution to the Cauchy problem \eqref{mBO}--\eqref{mBO-id} is unique in $C([0,T];H^s(\T ))$ for $s>1/2$.
\end{thm}
\begin{rem}
The uniqueness assertion also holds in the class $L^\I((0,T);H^s(\T ))$, $s>1/2$.
See \cite[Remark~1.2]{K-BOp} for a detailed argument.
\end{rem}

Our proof of unconditional uniqueness is based on the so-called \emph{normal form reduction} method.
This simple technique of gaining regularity from non-resonant nonlinear interactions by integration by parts in time became recognized as a useful tool to establish unconditional uniqueness for nonlinear dispersive equations; see, e.g., \cite{BIT11,GKO13,KOYp,K-all} and references therein.

It is worth comparing the above result with our previous works~\cite{K-all,K-BOp} on unconditional uniqueness for the periodic derivative nonlinear Schr\"odinger equation
\eq{DNLS}{\p _tu=i\p _x^2u+\p _x(|u|^2u),\qquad u:[0,T]\times \T \to \Bo{C}}
and the periodic Benjamin-Ono equation
\eq{BO}{\p _tu=-\H \p _x^2u+\p _x(u^2),\qquad u:[0,T]\times \T \to \R .}
First, a common idea among these three results is to use a gauge transform to eliminate unfavorable nonlinear interactions (called high-low type) which causes serious derivative loss.
For \eqref{DNLS} the gauge transform is simple and the transformed equation contains the transformed unknown function only, whereas the gauge transforms used for \eqref{BO} and \eqref{mBO} are much more complicated and one cannot remove the original unknown function from the transformed equation.
In fact, for \eqref{DNLS} an analogous result was obtained in the non-periodic case in \cite{MYp}, while for \eqref{BO} and \eqref{mBO} it is seemingly not clear whether the normal form approach has a non-periodic counterpart.

Secondly, the result for \eqref{BO} was proved by analyzing a normal form equation obtained after \emph{finite} (twice, actually) reduction steps, which is a similar situation to the periodic Korteweg-de~Vries equation studied in \cite{BIT11}, whereas \eqref{DNLS} and \eqref{mBO} require \emph{infinite} normal form reduction steps.
In the context of unconditional uniqueness, such infinite implementation of normal form reduction was initiated in~\cite{GKO13}.
In~\cite{K-all} we proposed an abstract framework of this strategy, which allows us to generate all the multilinear estimates of arbitrarily high degree by simply iterating certain \emph{fundamental multilinear estimates} of the lowest degree.
Our proof of Theorem~\ref{thm} basically follows the argument in~\cite{K-all}, but some technical modifications are needed due to complication of the equation after the gauge transform.
Although we estimate only the $C([0,T];H^s)$-norm of solutions and do not use Bourgain-type norms, the basic idea of the proof of fundamental estimates (Proposition~\ref{prop:fundamental} below) is quite similar to that of \cite{GLM14}. 

One may expect that the \emph{infinite} normal form reduction would also yield an improvement of  the result on \eqref{BO}.
However, it seems that the approach in this article based on the strategy in~\cite{K-all} shows unconditional uniqueness for \eqref{BO} in $H^s(\T )$ only for $s>1/2$, which is much worse than the result $s>1/6$ obtained in~\cite{K-BOp}.
The reason is that we only use half of the possible gain of derivative in the fundamental estimate (see the denominator in the estimate \eqref{est:matome-1} in Proposition~\ref{prop:fundamental}), which is however essential for maintaining infinite normal form reduction machinery.
In~\cite{K-BOp} we could only apply the reduction twice, but instead use full of gain of derivative.

The plan of this article is as follows.
The proof begins with applying a suitable gauge transform, which is similar to that of \cite{GLM14} and described in Section~\ref{sec:gauge}.
In Section~\ref{sec:nf}, the infinite normal form reduction scheme is formulated, and we use it to reduce the proof of Theorem~\ref{thm} to showing fundamental quintilinear estimates given in Proposition~\ref{prop:fundamental}.
Finally, we prove these estimates in Section~\ref{sec:proof-multi}. 

\textbf{Notations}.
For a Banach space $X$, we abbreviate $C([0,T];X)$ to $C_TX$.

The Fourier coefficients of a $2\pi$-periodic function $f$ are defined by 
\[ \F f(n):=\frac{1}{2\pi} \int _0^{2\pi}f(x)e^{-inx}\,dx,\qquad n\in \Bo{Z},\]
so that the inverse Fourier transform of a sequence $g=\{ g(n)\}_{n\in \Bo{Z}}$ is given by
\[ \F^{-1}g(x):=\sum _{n\in \Bo{Z}}g(n)e^{inx},\qquad x\in \Bo{T}.\]

For $N>0$, let $P_{\le N}:=\F ^{-1}\chf{|n|\le N}\F$ be the projection onto frequencies $\{ n\in \Bo{Z}:|n|\le N\}$, and $P_{>N}:=1-P_{\le N}$.
We use the notations $P_{\pm}:=\F ^{-1}\chf{\pm n>0}\F$, so that $\H =-iP_++iP_-$, and
\eqs{P_{c}f:=\F ^{-1}\chf{n=0}\F f=\F f(0)=\frac{1}{2\pi}\int _0^{2\pi}f(x)\,dx,\\
P_{\neq c}f:=\F ^{-1}\chf{n\neq 0}\F f=f-\F f(0).
}
For a $2\pi$-periodic function $f(x)$ with zero mean value (\mbox{i.e.} $P_{c}f=0$), define its periodic primitive
\eqq{\p _x^{-1}f(x):=\F ^{-1}\Big[ \frac{1}{in}\F f(n)\Big] (x)=\frac{1}{2\pi}\int _0^{2\pi}\int _{\th}^xf(y)\,dy\,d\th .}
Note that $\bbar{P_{\pm}f}=P_{\mp}\bar{f}$ and $\p _x^{-1}\p _x=\p _x\p _x^{-1}P_{\neq c}=P_{\neq c}$.

We often use the abbreviation $n_{ijk\dots}$ to mean $n_i+n_j+n_k+\cdots $; for instance, $n_{12}=n_1+n_2$ and $n-n_{235}=n-(n_2+n_3+n_5)$.

Finally, we use the weighted $\ell^p$ space; for $s\in \R$ and $1\le p\le \I$,
\[ \ell ^p_s(\Bo{Z}):=\Big\{ \om :\Bo{Z}\to \Bo{C} ~\Big| ~\norm{\om}{\ell ^p_s}:=\norm{\LR{\cdot}^s\om}{\ell ^p}<\infty \Big\} ,\qquad \LR{n}:=(1+n^2)^{1/2}.\]


\bigskip
\section{Gauge transform}\label{sec:gauge}

\smallskip
The mBO equation \eqref{mBO} is transformed into the equation
\eq{mBO'}{\p _tu=-\H \p _x^2u+2\sgm P_{\neq c}(u^2)\p _xu,\qquad (t,x)\in (0,T)\times \T}
by the change of the unknown function
\eqq{u(t,x)\quad \mapsto \quad 2^{-1/2}u\big( t,x-\sgm \int _0^tP_{c}(u(s)^2)\,ds\big) .}
In what follows we consider uniqueness of the solution to the Cauchy problem \eqref{mBO'}--\eqref{mBO-id} in $C_TH^s$, $s>1/2$.
This is sufficient for our purpose, since the above transform is a homeomorphism on $C_TH^s$ for $s\ge 0$.


\smallskip
\subsection{Definition of the gauge transform}

Let $s>1/2$ and $u\in C_TH^s$ be a real-valued solution of \eqref{mBO'} in the sense of distribution.
The nonlinear term $P_{\neq c}(u^2)\p _xu$ makes sense%
\footnote{One can in fact make sense of the nonlinearity for $s\ge 1/6$ by the identity $P_{\neq c}(u^2)\p _xu=\frac{1}{3}\p _x(u^3)-P_{c}(u^2)\p _xu$ and the embedding $H^{1/6}\hookrightarrow L^3$.} 
in $C_TH^{s-1}$ by the estimate
\eq{prod-est1}{\norm{fg}{H^{s-1}}\lec \norm{f}{H^s}\norm{g}{H^{s-1}},\qquad s>1/2,}
which is a special case of the Sobolev multiplication law (see, e.g., \cite[Lemma~3.4]{GLM14}).
In particular, $\p _tu\in C_TH^{s-2}$ and each Fourier coefficient $\F [u(t,\cdot )] (n)$ is a $C^1$ function of $t$ on $[0,T]$.
Testing the equation against constant functions in $x$, we see that the spatial mean of $u$ is independent of $t$:
\eqq{\nu :=P_cu(t)=\frac{1}{2\pi}\int _0^{2\pi}u(t,x)\,dx\equiv P_cu(0),\qquad t\in [0,T].}

We define the gauge transform as follows:
\begin{defn}
Let $F[f]:=\p _x^{-1}P_{\neq c}(f^2)$ for $f\in L^2(\T )$.
For $s>1/2$ and a real-valued function $u\in C_TH^s$, define its gauge transform $v$ by
\eq{utov}{v(t,x):=e^{-i\sgm F[u(t)](x)}\Big( \big[ P_+u(t)\big] (x)+\frac{1}{2}P_cu(t)\Big) .}
\end{defn}
The above definition is slightly different from that used in \cite{GLM14}, though it is based on the same idea.
Note that
\eq{u_via_v}{P_+u+\tfrac{1}{2}P_cu=e^{i\sgm F[u]}v,&\qquad u=e^{i\sgm F[u]}v+e^{-i\sgm F[u]}\bar{v},\\
\p _xP_+u=\p _xP_+\big( e^{i\sgm F[u]}v\big) ,&\qquad \p _xP_-u=\p _xP_-\big( e^{-i\sgm F[u]}\bar{v}\big) .}

The following two lemmas were essentially shown in \cite{GLM14}, but we give them with proofs.
First, we see that the gauge $e^{-iF[u]}$ has one higher regularity than the solution $u$:
\begin{lem}\label{lem:exp}
Let $1/2<s\le 1$.
For real-valued $f\in H^s$ and $k\in \R$, we have $e^{ikF[f]}\in H^{s+1}$.
Moreover, for real-valued $u,\ti{u}\in C_TH^s$ let
\eqq{X_s(u):=\max _{k\in \shugo{\pm 1,\pm 3}}\norm{e^{ikF[u]}}{C_TH^s},\qquad
Y_s(u,\ti{u}):=\max _{k\in \shugo{\pm 1,\pm 3}}\norm{e^{ikF[u]}-e^{ikF[\ti{u}]}}{C_TH^s},
}
then the following estimates hold:
\eqq{X_s(u)&\lec 1+\tnorm{u}{C_TH^s}^2,\\
X_{s+1}(u)&\lec 1+\tnorm{u}{C_TH^s}^4,\\
Y_s(u,\ti{u})&\lec (1+\tnorm{u}{C_TH^s}^3+\tnorm{\ti{u}}{C_TH^s}^3)\tnorm{u-\ti{u}}{C_TH^s},\\
Y_{s+1}(u,\ti{u})&\lec (1+\tnorm{u}{C_TH^s}^5+\tnorm{\ti{u}}{C_TH^s}^5)\tnorm{u-\ti{u}}{C_TH^s}.
}
In particular, the gauge transform \eqref{utov} is locally Lipschitz continuous from $C_TH^s(\T ;\R)$ into $C_TH^s(\T ;\Bo{C})$ with estimates
\eqq{\norm{v}{C_TH^s}&\lec  (1+\norm{u}{C_TH^s}^2)\norm{u}{C_TH^s},\\
\norm{v-\ti{v}}{C_TH^s}&\lec  (1+\tnorm{u}{C_TH^s}^4+\tnorm{\ti{u}}{C_TH^s}^4)\tnorm{u-\ti{u}}{C_TH^s}.}
\end{lem}

\begin{proof}
We have
\eqq{\norm{e^{ikF[f]}}{H^1}&\lec \norm{e^{ikF[f]}}{L^2}+\norm{ke^{ikF[f]}P_{\neq c}(f^2)}{L^2}\lec _k1+\norm{f}{H^s}^2,\\
\norm{e^{ikF[f]}-e^{ikF[g]}}{H^1}&\lec \norm{e^{ikF[f]}-e^{ikF[g]}}{L^2}+\norm{ke^{ikF[f]}P_{\neq c}(f^2)-ke^{ikF[g]}P_{\neq c}(g^2)}{L^2}\\
&\lec _k\norm{F[f]-F[g]}{L^2}\big( 1+\norm{f^2}{L^\I}\big) +\norm{f^2-g^2}{L^2}\\
&\lec \big( 1+\norm{f}{H^s}^3+\norm{g}{H^s}^3\big) \norm{f-g}{H^s},}
and thus
\eqqed{\norm{e^{ikF[f]}}{H^{s+1}}&\lec 1+\norm{ke^{ikF[f]}P_{\neq c}(f^2)}{H^s}\lec _k1+(1+\norm{f}{H^s}^2)\norm{f}{H^s}^2\lec 1+\norm{f}{H^s}^4,\\
\norm{e^{ikF[f]}-e^{ikF[g]}}{H^{s+1}}&\lec \norm{e^{ikF[f]}-e^{ikF[g]}}{L^2}+\norm{ke^{ikF[f]}P_{\neq c}(f^2)-ke^{ikF[g]}P_{\neq c}(g^2)}{H^s}\\
&\lec _k\big( 1+\norm{f}{H^s}^5+\norm{g}{H^s}^5\big) \norm{f-g}{H^s}.}
\end{proof}

Using higher regularity of the gauge part shown above, we can `invert' the gauge transform on solutions of \eqref{mBO'} for a short time:
\begin{lem}\label{lem:est-u}
Let $1/2<s\le 1$, $u,\ti{u}\in C_TH^s$ be real-valued solutions of \eqref{mBO'} and $v,\ti{v}$ be the corresponding gauge transforms defined by \eqref{utov}.
Then, there exist $C>0$ depending on $s, \norm{u}{C_TH^s}, \norm{\ti{u}}{C_TH^s}$ and $T_0\in (0,T]$ depending also on the profile of $v$ in $C_TH^s$, such that
\eqq{\norm{u-\ti{u}}{C_{T'}H^s}\le C\big( \norm{u(0)-\ti{u}(0)}{H^s}+\norm{v-\ti{v}}{C_{T'}H^s}\big)}
for any $T'\in (0,T_0]$.
\end{lem}

\begin{proof}
With a constant $N\gg 1$ to be chosen later, we first divide the norm as
\eqq{\norm{u-\ti{u}}{C_TH^s}\le \norm{P_{\le 2N}\big( u-\ti{u}\big)}{C_TH^s}+\norm{P_{>2N}\big( u-\ti{u}\big)}{C_TH^s}.}
Since $P_{\le 2N}u$ and $P_{\le 2N}\ti{u}$ are smooth and satisfy the integral equation
\eqq{P_{\le 2N}u(t)=P_{\le 2N}e^{-t\H \p _x^2}u(0)+2\int _0^te^{-(t-t')\H \p _x^2}P_{\le 2N}\big[ P_{\neq c}(u(t')^2)\p _xu(t')\big] \,dt'}
in the classical sense, we get
\eqq{\norm{P_{\le 2N}\big( u-\ti{u}\big)}{C_TH^s}&\le \norm{u(0)-\ti{u}(0)}{H^s}+2\int _0^T\norm{P_{\le 2N}\big[ P_{\neq c}(u^2)\p _xu-P_{\neq c}(\ti{u}^2)\p _x\ti{u}\big] (t)}{H^s}\,dt\\
&\le \norm{u(0)-\ti{u}(0)}{H^s}+CTN\norm{P_{\neq c}(u^2)\p _xu-P_{\neq c}(\ti{u}^2)\p _x\ti{u}}{C_TH^{s-1}}\\
&\le \norm{u(0)-\ti{u}(0)}{H^s}+CTN\big( \tnorm{u}{C_TH^s}^2+\tnorm{\ti{u}}{C_TH^s}^2\big) \tnorm{u-\ti{u}}{C_TH^s},}
where at the last step we have used the product estimate \eqref{prod-est1}.
For the estimate in high frequencies, we use the first identity in \eqref{u_via_v} and obtain
\eqq{&\norm{P_{>2N}\big( u-\ti{u}\big)}{C_TH^s}\le 2\norm{P_{>2N}P_+\big( u-\ti{u}\big)}{C_TH^s}=2\norm{P_{>2N}\big( e^{i\sgm F[u]}v-e^{i\sgm F[\ti{u}]}\ti{v}\big)}{C_TH^s}\\
&\le 2\norm{P_{>N}(e^{i\sgm F[u]}-e^{i\sgm F[\ti{u}]})\cdot v}{C_TH^s}+2\norm{P_{\le N}(e^{i\sgm F[u]}-e^{i\sgm F[\ti{u}]})\cdot P_{>N}v}{C_TH^s}\\
&\hx +2\norm{e^{i\sgm F[\ti{u}]}(v-\ti{v})}{C_TH^s}\\
&\le C\Big( N^{-1}Y_{s+1}(u,\ti{u})\norm{v}{C_TH^s}+Y_s(u,\ti{u})\norm{P_{>N}v}{C_TH^s}+X_s(\ti{u})\norm{v-\ti{v}}{C_TH^s}\Big) .}
Using Lemma~\ref{lem:exp}, we have
\eqq{\norm{u-\ti{u}}{C_TH^s}&\le \norm{u(0)-\ti{u}(0)}{H^s}+C\LR{\norm{\ti{u}}{C_TH^s}}^2\norm{v-\ti{v}}{C_TH^s}\\[-5pt]
&\hx +C\Big\{ TN(\norm{u}{C_TH^s}+\norm{\ti{u}}{C_TH^s})^2+N^{-1}\LR{\norm{u}{C_TH^s}+\norm{\ti{u}}{C_TH^s}}^8\\[-5pt]
&\hx \hspace{25pt} +\LR{\norm{u}{C_TH^s}+\norm{\ti{u}}{C_TH^s}}^3\norm{P_{>N}v}{C_TH^s}\Big\} \tnorm{u-\ti{u}}{C_TH^s}.}
Note that $\norm{P_{>N}v}{C_TH^s}\to 0$ as $N\to \I$.
We take $N\gg 1$ so that
\eqq{C\Big\{ N^{-1}\LR{\norm{u}{C_TH^s}+\norm{\ti{u}}{C_TH^s}}^8+\LR{\norm{u}{C_TH^s}+\norm{\ti{u}}{C_TH^s}}^3\norm{P_{>N}v}{C_TH^s}\Big\} \le 1/3,}
and then take $0<T_0\le T$ so that 
\eqq{CT_0N(\norm{u}{C_TH^s}+\norm{\ti{u}}{C_TH^s})^2\le 1/3.}
Repeating the above argument with $T$ replaced by any $T'\in (0,T_0]$, we obtain 
\eqq{\norm{u-\ti{u}}{C_{T'}H^s}&\le 3\big( \norm{u(0)-\ti{u}(0)}{H^s}+C\LR{\norm{\ti{u}}{C_{T'}H^s}}^2\norm{v-\ti{v}}{C_{T'}H^s}\big) .\qedhere}
\end{proof}


\smallskip
\subsection{Equation after the gauge transform}

In view of Lemma~\ref{lem:est-u}, it suffices to show uniqueness for the equation satisfied by the gauge transform $v$ of the solution $u$ of \eqref{mBO'}.
It is easy to derive the equation for $v$ \emph{formally}, assuming that the solution $u$ of \eqref{mBO'} is sufficiently smooth.
However, in order to show unconditional uniqueness we need to consider general solutions of \eqref{mBO'} which may not be approximated by smooth solutions.
Hence, given a solution $u\in C_TH^s$ of \eqref{mBO'}, we first set
\eqq{v_N(t,x):=e^{-i\sgm F_N(t,x)}\Big( \big[ P_+u_N(t)\big] (x)+\frac{\nu}{2}\Big) ,\qquad F_N:=F[u_N],\quad u_N:=P_{\le N}u}
and consider the equation for $v_N$.
(Recall that $\nu =P_cu(t)$ is conserved.)
Observing that the equalities
\eqq{\p _tu_N&=-\H \p _x^2u_N+2\sgm P_{\neq c}(u_N^2)\p _xu_N+G_N,\\
-i\sgm \p _tF_N&=-i\sgm \p _x^{-1}P_{\neq c}\Big( 2u_N\big( -\H \p _x^2u_N+2\sgm P_{\neq c}(u_N^2)\p _xu_N+G_N\big) \Big) \\
&=2\sgm P_{\neq c}(u_Ni\H \p _xu_N)-2\sgm B(u_N,u_N)-iP_{\neq c}\big[ \big( P_{\neq c}(u_N^2)\big) ^2\big] -2i\sgm \p _x^{-1}P_{\neq c}(u_NG_N),\\
\p _xF_N&=P_{\neq c}(u_N^2),\qquad \p _x^2F_N=2u_N\p _xu_N}
hold in the classical sense, where
\eqq{G_N&:=2\sgm P_{\le N}\big( P_{\neq c}(u^2)\p _xu\big) -2\sgm P_{\neq c}(u_N^2)\p _xu_N,\\
B(f,g)&:=\p _x^{-1}\big( (P_+\p _xf)(P_+\p _xg)-(P_-\p _xf)(P_-\p _xg)\big) ,}
we see that
\eqq{&(\p _t+\H \p _x^2)v_N=e^{-i\sgm F_N}\Big( -i\sgm \p _tF_N\cdot \big( P_+u_N+\tfrac{\nu}{2}\big) +P_+\p _tu_N\Big) 
\\
&\hx\hx +\H \Big[ e^{-i\sgm F_N}\Big( \big( -i\sgm \p _x^2F_N-(\p _xF_N)^2\big) \big(P_+u_N+\tfrac{\nu}{2}\big) -2i\sgm \p _xF_N\cdot P_+\p _xu_N+P_+\p _x^2u_N\Big) \Big] \\
&\hx =e^{-i\sgm F_N}\Big( 2\sgm P_{\neq c}(u_Ni\H \p _xu_N)-2\sgm B(u_N,u_N)-iP_{\neq c}\big[ \big( P_{\neq c}(u_N^2)\big) ^2\big] -2i\sgm \p _x^{-1}P_{\neq c}(u_NG_N)\Big) \\
&\hspace{350pt} \times \big( P_+u_N+\tfrac{\nu}{2}\big) \\
&\hx\hx +e^{-i\sgm F_N}\Big( iP_+\p _x^2u_N+2\sgm P_+\big( P_{\neq c}(u_N^2)\p _xu_N\big) +P_+G_N\Big) \\
&\hx\hx +i\H \Big[ e^{-i\sgm F_N}\Big\{ \Big( -2\sgm u_N\p _xu_N+i\big( P_{\neq c}(u_N^2)\big) ^2\Big) \big(P_+u_N+\tfrac{\nu}{2}\big) \\
&\hspace{250pt} -2\sgm P_{\neq c}(u_N^2)P_+\p _xu_N-iP_+\p _x^2u_N\Big\} \Big] \\
&\hx =2\sgm I_1+2\sgm I_2+I_3 +I_4,
}
\eqq{I_1&:=e^{-i\sgm F_N}u_N\big( P_+u_N+\tfrac{\nu}{2}\big) \big( P_+-P_-\big) \p _xu_N-(P_+-P_-)\Big[ e^{-i\sgm F_N}u_N\big( P_+u_N+\tfrac{\nu}{2}\big) \p _xu_N\Big] ,\\
I_2&:=e^{-i\sgm F_N}P_+\big( P_{\neq c}(u_N^2)\p _xu_N\big) -(P_+-P_-)\Big[ e^{-i\sgm F_N}P_{\neq c}(u_N^2)P_+\p _xu_N\Big] ,\\
I_3&:=ie^{-i\sgm F_N}P_+\p _x^2u_N-i(P_+-P_-)\Big[ e^{-i\sgm F_N}P_+\p _x^2u_N\Big] ,\\
I_4&:=e^{-i\sgm F_N}\Big( -2\sgm B(u_N,u_N)-iP_{\neq c}\big[ \big( P_{\neq c}(u_N^2)\big) ^2\big] -2i\sgm \p _x^{-1}P_{\neq c}(u_NG_N)\Big) \big( P_+u_N+\tfrac{\nu}{2}\big) \\
&\hx +e^{-i\sgm F_N}P_+G_N-\H \Big[ e^{-i\sgm F_N}\big( P_{\neq c}(u_N^2)\big) ^2\big( P_+u_N+\tfrac{\nu}{2}\big) \Big] \\
&\hx -2\sgm P_{c}\big( u_Ni\H \p _xu_N\big) e^{-i\sgm F_N}\big( P_+u_N+\tfrac{\nu}{2}\big) .
}
For $I_1$, $I_2$, and $I_3$, we have 
\eqq{I_1&=-2P_+\Big[ e^{-i\sgm F_N}u_N\big( P_+u_N+\tfrac{\nu}{2}\big) \p _xP_-u_N\Big] +2P_-\Big[ e^{-i\sgm F_N}u_N\big( P_+u_N+\tfrac{\nu}{2}\big) \p _xP_+u_N\Big] \\
&\hx +P_{c}\Big[ e^{-i\sgm F_N}u_N\big( P_+u_N+\tfrac{\nu}{2}\big) i\H \p _xu_N\Big] ,\\
I_2&=(P_c+P_-)\Big[ e^{-i\sgm F_N}P_+\big( P_{\neq c}(u_N^2)\p _xu_N\big) \Big] -P_+\Big[ e^{-i\sgm F_N}(P_c+P_-)\big( P_{\neq c}(u_N^2)\p _xu_N\big) \Big] \\
&\hx +P_+\Big[ e^{-i\sgm F_N}u_N^2\p _xP_-u_N\Big] +P_-\Big[ e^{-i\sgm F_N}u_N^2\p _xP_+u_N\Big] \\
&\hx -P_c(u_N^2)\Big( P_+\big[ e^{-i\sgm F_N}\p _xP_-u_N\big] +P_-\big[ e^{-i\sgm F_N}\p _xP_+u_N\big] \Big) ,\\
I_3&=iP_c\big[ e^{-i\sgm F_N}P_+\p _x^2u_N\big] +2iP_-\big[ e^{-i\sgm F_N}P_+\p _x^2u_N\big] \\
&=iP_c\big[ e^{-i\sgm F_N}P_+\p _x^2u_N\big] +2\sgm P_-\big[ \p _x^{-1}P_-\big( e^{-i\sgm F_N}u_N^2\big) \cdot \p _x^2P_+u_N\big] \\
&\hx -2\sgm P_{c}(u_N^2)P_-\big[ \p _x^{-1}P_-\big( e^{-i\sgm F_N}\big) \cdot \p _x^2P_+u_N\big] \\
&=2\sgm \p _xP_-\big[ \p _x^{-1}P_-\big( e^{-i\sgm F_N}u_N^2\big) \cdot \p _xP_+u_N\big] -2\sgm P_-\big[ e^{-i\sgm F_N}u_N^2\p _xP_+u_N\big] \\
&\hx +iP_c\big[ e^{-i\sgm F_N}P_+\p _x^2u_N\big] -2\sgm P_{c}(u_N^2)P_-\big[ \p _x^{-1}P_-\big( e^{-i\sgm F_N}\big) \cdot \p _x^2P_+u_N\big] .
}
Combining these identities, we have 
\eq{eq:v_N}{(\p _t\!+\!\H \p _x^2)v_N&=-4\sgm P_+\Big[ e^{-i\sgm F_N}u_N\big( P_+u_N\!+\!\tfrac{\nu}{2}\big) \p _xP_-u_N\Big] +4\sgm P_-\Big[ e^{-i\sgm F_N}u_N\big( P_+u_N\!+\!\tfrac{\nu}{2}\big) \p _xP_+u_N\Big] \\
&\hx +2\sgm P_+\Big[ e^{-i\sgm F_N}u_N^2\p _xP_-u_N\Big] +2\sgm \p _xP_-\big[ \p _x^{-1}P_-\big( e^{-i\sgm F_N}u_N^2\big) \cdot \p _xP_+u_N\big] \\
&\hx -2\sgm e^{-i\sgm F_N}\big( P_+u_N+\tfrac{\nu}{2}\big) B(u_N,u_N)+R_N[u]}
in the classical sense, where
\eqq{R_N[u]&:=e^{-i\sgm F_N}\Big\{ \Big( -iP_{\neq c}\big[ \big( P_{\neq c}(u_N^2)\big) ^2\big] -2i\sgm \p _x^{-1}P_{\neq c}(u_NG_N)\Big) \big( P_+u_N+\tfrac{\nu}{2}\big) +P_+G_N\Big\} \\
&\hx -\H \Big[ e^{-i\sgm F_N}\big( P_{\neq c}(u_N^2)\big) ^2\big( P_+u_N+\tfrac{\nu}{2}\big) \Big] -2\sgm P_{c}\big( u_Ni\H \p _xu_N\big) e^{-i\sgm F_N}\big( P_+u_N+\tfrac{\nu}{2}\big) \\
&\hx +2\sgm P_{c}\Big[ e^{-i\sgm F_N}u_N\big( P_+u_N+\tfrac{\nu}{2}\big) i\H \p _xu_N\Big] \\
&\hx +2\sgm (P_c+P_-)\Big[ e^{-i\sgm F_N}P_+\big( P_{\neq c}(u_N^2)\p _xu_N\big) \Big] -2\sgm P_+\Big[ e^{-i\sgm F_N}(P_c+P_-)\big( P_{\neq c}(u_N^2)\p _xu_N\big) \Big] \\
&\hx -2\sgm P_c(u_N^2)\Big( P_+\big[ e^{-i\sgm F_N}\p _xP_-u_N\big] +P_-\big[ e^{-i\sgm F_N}\p _xP_+u_N\big] \Big) \\
&\hx +iP_c\big[ e^{-i\sgm F_N}P_+\p _x^2u_N\big] -2\sgm P_{c}(u_N^2)P_-\big[ \p _x^{-1}P_-\big( e^{-i\sgm F_N}\big) \cdot \p _x^2P_+u_N\big] .
}

To take the limit $N\to \I$ in \eqref{eq:v_N}, we prepare one more estimate on the bilinear form $B(f,g)$ which is easily deduced from \eqref{prod-est1}:
\eq{prod-est2}{\norm{B(f,g)}{H^{s-1}}\lec \norm{f}{H^s}\norm{g}{H^s},\qquad s>1/2.}
We also note that $G_N\to 0$ in $C_TH^{s-1}$ as $N\to \I$, which also follows from \eqref{prod-est1}.
Exploiting \eqref{prod-est1} and \eqref{prod-est2}, and Lemma~\ref{lem:exp}, we can take the limit of the right-hand side of \eqref{eq:v_N} in $C_TH^{s-1}$, and its left-hand side in the sense of distribution, obtaining the equation for $v$ as
\eqq{(\p _t+\H \p _x^2)v&=-4\sgm P_+\Big[ e^{-i\sgm F[u]}u\big( P_+u+\tfrac{\nu}{2}\big) \p _xP_-u\Big] +4\sgm P_-\Big[ e^{-i\sgm F[u]}u\big( P_+u+\tfrac{\nu}{2}\big) \p _xP_+u\Big] \\*
&\hx +2\sgm P_+\Big[ e^{-i\sgm F[u]}u^2\p _xP_-u\Big] +2\sgm \p _xP_-\big[ \p _x^{-1}P_-\big( e^{-i\sgm F[u]}u^2\big) \cdot \p _xP_+u\big] \\*
&\hx -2\sgm e^{-i\sgm F[u]}\big( P_+u+\tfrac{\nu}{2}\big) \p _x^{-1}\big( (P_+\p _xu)^2-(P_-\p _xu)^2\big) +R[u],}
where
\eq{def:R}{R[u]&:=-ie^{-i\sgm F[u]}P_{\neq c}\big[ \big( P_{\neq c}(u^2)\big) ^2\big] \big( P_+u+\tfrac{\nu}{2}\big) -\H \Big[ e^{-i\sgm F[u]}\big( P_{\neq c}(u^2)\big) ^2\big( P_+u+\tfrac{\nu}{2}\big) \Big] \\
&\hx -2\sgm P_{c}\big( ui\H \p _xu\big) e^{-i\sgm F[u]}\big( P_+u+\tfrac{\nu}{2}\big) +2\sgm P_{c}\Big[ e^{-i\sgm F[u]}u\big( P_+u+\tfrac{\nu}{2}\big) i\H \p _xu\Big] \\
&\hx +2\sgm (P_c+P_-)\Big[ e^{-i\sgm F[u]}P_+\big( P_{\neq c}(u^2)\p _xu\big) \Big] -2\sgm P_+\Big[ e^{-i\sgm F[u]}(P_c+P_-)\big( P_{\neq c}(u^2)\p _xu\big) \Big] \\
&\hx -2\sgm P_c(u^2)\Big( P_+\big[ e^{-i\sgm F[u]}\p _xP_-u\big] +P_-\big[ e^{-i\sgm F[u]}\p _xP_+u\big] \Big) \\
&\hx +iP_c\big[ e^{-i\sgm F[u]}P_+\p _x^2u\big] -2\sgm P_{c}(u^2)P_-\big[ \p _x^{-1}P_-\big( e^{-i\sgm F[u]}\big) \cdot \p _x^2P_+u\big] .
}
Substituting \eqref{u_via_v} in the above equation, we finally obtain 
\eq{eq:v}{(\p _t+\H \p _x^2)v&=-2\sgm P_+\Big[ e^{i\sgm F[u]}v^2\p _xP_-(e^{-i\sgm F[u]}\bar{v})\Big] +2\sgm P_+\Big[ e^{-3i\sgm F[u]}\bar{v}^2\p _xP_-(e^{-i\sgm F[u]}\bar{v})\Big] \\
&\hx +4\sgm P_-\Big[ e^{i\sgm F[u]}v^2\p _xP_+(e^{i\sgm F[u]}v)\Big] +4\sgm P_-\Big[ e^{-i\sgm F[u]}v\bar{v}\p _xP_+(e^{i\sgm F[u]}v)\Big] \\
&\hx +2\sgm \p _xP_-\Big[ \p _x^{-1}P_-\big( e^{i\sgm F[u]}v^2+2e^{-i\sgm F[u]}v\bar{v}+e^{-3i\sgm F[u]}\bar{v}^2\big) \cdot \p _xP_+(e^{i\sgm F[u]}v)\Big] \\
&\hx -2\sgm v\p _x^{-1}\Big[ \big( \p _xP_+(e^{i\sgm F[u]}v)\big) ^2\Big] +2\sgm v\p _x^{-1}\Big[ \big( \p _xP_-(e^{-i\sgm F[u]}\bar{v})\big) ^2\Big] +R[u],}
with $R[u]$ given by \eqref{def:R}.

It is easy to see that all the terms on the right-hand side of \eqref{eq:v} belong to $C_TH^{s-1}$ as long as $u$ is in $C_TH^s$ with $1/2<s\le 1$ and so is $v$. 
Moreover, the remainder $R[u]$ is in $C_TH^s$; in fact, from Lemma~\ref{lem:exp} we see
\eq{est:Ru}{\norm{R[u]}{C_TH^s}&\lec X_s(u)\big( \tnorm{u}{C_TH^s}^3+\tnorm{u}{C_TH^s}^5\big) +X_{s+1}(u)\big( \tnorm{u}{C_TH^s}+\tnorm{u}{C_TH^s}^3\big)\\
&\lec \big( 1+\norm{u}{C_TH^s}^6\big) \norm{u}{C_TH^s},\\
\norm{R[u]-R[\ti{u}]}{C_TH^s}&\lec Y_s(u,\ti{u})\big( \tnorm{u}{C_TH^s}^3+\tnorm{\ti{u}}{C_TH^s}^3+\tnorm{u}{C_TH^s}^5+\tnorm{\ti{u}}{C_TH^s}^5\big) \\
&\quad +\big( X_{s}(u)\!+\!X_{s}(\ti{u})\big) \big( \tnorm{u}{C_TH^s}^2\!+\!\tnorm{\ti{u}}{C_TH^s}^2\!+\!\norm{u}{C_TH^s}^4\!+\!\tnorm{\ti{u}}{C_TH^s}^4\big) \tnorm{u\!-\!\ti{u}}{C_TH^s}\\
&\quad +Y_{s+1}(u,\ti{u})\big( \tnorm{u}{C_TH^s}+\tnorm{\ti{u}}{C_TH^s}+\tnorm{u}{C_TH^s}^3+\tnorm{\ti{u}}{C_TH^s}^3\big) \\
&\quad +\big( X_{s+1}(u)+X_{s+1}(\ti{u})\big) \big( 1+\tnorm{u}{C_TH^s}^2+\tnorm{\ti{u}}{C_TH^s}^2\big) \tnorm{u-\ti{u}}{C_TH^s}\\
&\lec \big( 1+\norm{u}{C_TH^s}^8+\norm{\ti{u}}{C_TH^s}^8\big) \norm{u-\ti{u}}{C_TH^s}.
}

In \eqref{eq:v}, the original unknown function $u$ appears only in the gauge and the remainder parts.
Since these terms behave better than the others in view of Lemma~\ref{lem:exp} and \eqref{est:Ru}, the main part of the nonlinear interactions in \eqref{eq:v} essentially consists of $v$ only and is of the following types:
\eqq{P_\pm \big[ v\cdot v\cdot \p _xP_\mp v\big] ,\qquad \p _xP_-\big[ \p _x^{-1}(v\cdot v)\cdot \p _xP_+v\big] ,\qquad v\cdot \p _x^{-1}\big[ \p _xP_\pm v\big] ^2.}
In particular, there is no high-low type interaction in which the spatial derivative is put on the function of the highest frequency and cannot be moved to any other functions.


\bigskip
\section{Reduction to the fundamental quintilinear estimates}\label{sec:nf}

\smallskip
In this section, we apply the method of normal form reduction to the equation \eqref{eq:v} of $v$.
The a priori bound for a solution and difference of solutions will be reduced to certain fundamental quintilinear estimates (Proposition~\ref{prop:fundamental} below), which will be proved in Section~\ref{sec:proof-multi}.

\smallskip
\subsection{Equation in the Fourier side}
Let $1/2<s\le 1$ and $u,v\in C_TH^s$ be solutions of \eqref{mBO'} and \eqref{eq:v}, respectively, related through \eqref{utov}.
We first introduce a new unknown function
\eq{vtoom}{\om (t,n):=e^{itn|n|}\F [v(t,\cdot )](n)}
and derive the equation for $\om$.
We denote by $\Sc{Q}(m)$ the quintilinear form on $\ell ^2(\Bo{Z})$ associated with a multiplier $m(n,n_1,\dots ,n_5)$ on $\{ (n,n_1,\dots ,n_5)\in \Bo{Z}^6:n=n_{12345}\}$ defined by 
\eqq{\Sc{Q}\big( m;\;\om _1,W_2,\om _3,W_4,\om _5\big) (n):=\sum _{n=n_{12345}}m(n, n_1,\dots ,n_5)\om _1(n_1)W_2(n_2)\om _3(n_3)W_4(n_4)\om _5(n_5).}
With this notation, the equation for $\om$ reads as
\eqq{\p _t\om (t,n)&=\sum _{i=1}^7\Sc{Q}_i\big( e^{it\Phi}m_i;\;u(t),\om (t)\big) (n)+e^{itn|n|}\F (R[u])(t,n),}
where
\eqq{\Sc{Q}_1\big( m;\;u,\om \big) &:=\Sc{Q}\Big( m;\;\om ,\,\F (e^{i\sgm F[u]}),\,\om ,\,\F (e^{-i\sgm F[u]}),\,\om ^*\Big) ,\\
\Sc{Q}_2\big( m;\;u,\om \big) &:=\Sc{Q}\Big( m;\;\om ^*,\,\F (e^{-3i\sgm F[u]}),\,\om ^*,\,\F (e^{-i\sgm F[u]}),\,\om ^*\Big) ,\\
\Sc{Q}_3\big( m;\;u,\om \big) &:=\Sc{Q}\Big( m;\;\om ,\,\F (e^{i\sgm F[u]}),\,\om ,\,\F (e^{i\sgm F[u]}),\,\om \Big) ,\\
\Sc{Q}_4\big( m;\;u,\om \big) &:=\Sc{Q}\Big( m;\;\om ,\,\F (e^{-i\sgm F[u]}),\,\om ^*,\,\F (e^{i\sgm F[u]}),\,\om \Big) ,\\
\Sc{Q}_5\big( m;\;u,\om \big) &:=\Sc{Q}\Big( m;\;\om ^*,\,\F (e^{-3i\sgm F[u]}),\,\om ^*,\,\F (e^{i\sgm F[u]}),\,\om \Big) ,\\
\Sc{Q}_6\big( m;\;u,\om \big) &:=\Sc{Q}\Big( m;\;\om ,\,\F (e^{i\sgm F[u]}),\,\om ,\,\F (e^{i\sgm F[u]}),\,\om \Big) ,\\
\Sc{Q}_7\big( m;\;u,\om \big) &:=\Sc{Q}\Big( m;\;\om ,\,\F (e^{-i\sgm F[u]}),\,\om ^*,\,\F (e^{-i\sgm F[u]}),\,\om ^*\Big) ,}
and
\eqs{m_1:=c_1n_{45}\chf{n>0,\,n_{45}<0},\quad m_2:=c_2n_{45}\chf{n>0,\,n_{45}<0},\\
m_3:=c_3n_{45}\big( 2+\frac{n}{n_{123}}\big) \chf{n<0,\,n_{45}>0},~m_4:=c_4n_{45}\big( 1+\frac{n}{n_{123}}\big) \chf{n<0,\,n_{45}>0},~m_5:=c_5\frac{n_{45}n}{n_{123}}\chf{n<0,\,n_{45}>0},\\
m_6:=c_6\frac{n_{23}n_{45}}{n_{2345}}\chf{n_{23}>0,\,n_{45}>0},\quad m_7:=c_7\frac{n_{23}n_{45}}{n_{2345}}\chf{n_{23}<0,\,n_{45}<0},
}
$\Phi :=n|n|-n_1|n_1|-n_3|n_3|-n_5|n_5|$, $c_i$ ($1\le i\le 7$) are some constants, and
\eqq{\om ^*(t,n):=e^{itn|n|}\F \bar{v}(t,n)=\bbar{e^{it(-n)|-n|}\F v(t,-n)}=\bbar{\om (t,-n)}.}
Note that $|m_i|\lec n_{45}\chf{n<0,\,n_{45}>0}$ for $i=3,4,5$.

We also observe that the equation for $\om ^*$ is written as
\eqq{\p _t\om ^*(t,n)&=\sum _{i=1}^7\Sc{Q}_i^*\big( e^{it\Phi}m_i^*;\;u(t),\om (t)\big) (n)+e^{itn|n|}\F (\bbar{R[u]})(t,n),}
where $\Sc{Q}_i^*$ is defined by replacing $\om ,\om ^*, \F (e^{ikF[u]})$ ($k\in \shugo{\sgm ,-\sgm ,-3\sgm}$) with $\om ^*,\om ,\F (e^{-ikF[u]})$, respectively, in the definition of $\Sc{Q}_i$, and $m_i^*(n,n_1,\dots ,n_5):=\bbar{m_i(-n,-n_1,\dots ,-n_5)}$.

Next, we separate some harmless parts from $\Sc{Q}_i$.
Let $\eta >0$ be a small number, say $\eta =2^{-10}$, and let $\Sc{A}_1,\Sc{A}_2\subset \Bo{Z}^5$ be the sets of frequencies defined by
\eqq{\Sc{A}_1&:=\{ \,(n_l)_{l=1}^5: n_{15}n_{35}=0\,\} \\
&\qquad \cup ~\{ \,(n_l)_{l=1}^5: |n_2|\vee |n_4|\ge \eta ^2(|n_5|\wedge |n_{12345}|)\,\} \\
&\qquad \cup ~\{ \,(n_l)_{l=1}^5: |n_2|\vee |n_4|<\eta ^2|n_5|,\,|n_5|<\eta (|n_1|\wedge |n_3|),\,|n_{24}|\ge \eta |n_{13}|\,\} \\
&\qquad \cup ~\{ \,(n_l)_{l=1}^5: |n_2|\vee |n_4|<\eta ^2|n_5|,\,|n_3|<\eta (|n_1|\wedge |n_5|),\,|n_{24}|\ge \eta |n_{15}|,\,|n_5|\le 2|n_1|\,\} ,\\
\Sc{A}_2&:=\{ \, (n_l)_{l=1}^5 : n_{13}n_{15}=0~or~|n_2|\vee |n_4|\ge \eta (|n_3|\wedge |n_5|)\,\}.}
We divide the multipliers $m_i$ as $\hat{m}_{i}+\hat{\hat{m}}_{i}$,
\eqq{\hat{m}_i&:=\begin{cases}
m_i\chf{(n_l)_{l=1}^5\not\in \Sc{A}_1~and~|\Phi |>|n_2|^2\vee |n_4|^2} &\text{if $1\le i\le 5$},\\
m_i\chf{(n_l)_{l=1}^5\not\in \Sc{A}_2~and~|\Phi |>|n_2|^2\vee |n_4|^2} &\text{if $i=6,7$},
\end{cases}\\
\hat{\hat{m}}_i&:=\begin{cases}
m_i\chf{(n_l)_{l=1}^5\in \Sc{A}_1~or~|\Phi |\le |n_2|^2\vee |n_4|^2} &\text{if $1\le i\le 5$},\\
m_i\chf{(n_l)_{l=1}^5\in \Sc{A}_2~or~|\Phi |\le |n_2|^2\vee |n_4|^2} &\text{if $i=6,7$},
\end{cases}
}
and similarly, $m_i^*=\hat{m}_i^*+\hat{\hat{m}}_i^*$.
The equations for $\om,\om ^*$ are rewritten as
\eq{eq_om}{\p _t\om (t,n)&=\sum _{i=1}^7\Sc{Q}_i\big( e^{it\Phi}\hat{m}_{i};\;u(t),\om (t)\big) (n)+\Sc{R}^{(0)}[u,\om ](t,n),\\[-5pt]
\p _t\om ^*(t,n)&=\sum _{i=1}^7\Sc{Q}_i^*\big( e^{it\Phi}\hat{m}_{i}^*;\;u(t),\om (t)\big) (n)+\bbar{\Sc{R}^{(0)}[u,\om ](t,-n)},}
where
\eqq{\Sc{R}^{(0)}[u,\om ](t,n):=\sum _{i=1}^7\Sc{Q}_i\big( e^{it\Phi}\hat{\hat{m}}_{i};\;u(t),\om (t)\big) (n)+e^{itn|n|}\F (R[u])(t,n).}


\smallskip
\subsection{Definition of tree notation.}
Let us now define the notation of trees, which is a slight modification of that introduced in \cite{GKO13}.
\begin{defn}[Trees]
For $J\in \Bo{N}$, a tree $\Sc{T}$ of the $J$-th generation means a partially ordered set (with a partial order $\succeq$) of $5J+1$ elements satisfying all the following conditions.
\begin{enumerate}
\item There exists only one element $a\in \Sc{T}$ called the root such that $a\succeq b$ for all $b\in \Sc{T}$.
\item  Each element of $\Sc{T}$ except for the root has exactly one parent, where we say that $a\in \Sc{T}$ is the parent of $b\in \Sc{T}$ and $b$ is a child of $a$ if $a\neq b$, $a\succeq b$, and $a\succeq c\succeq b$ implies $c=a$ or $c=b$.
(We write $\Sc{T}^0$ to denote the subset of $\Sc{T}$ consisting of all elements that have a child, and $\Sc{T}^\I :=\Sc{T}\setminus \Sc{T}^0$.)
\item Each element of $\Sc{T}^0$ has exactly five children.
(This condition implies that $\# \Sc{T}^0=J$ and $\# \Sc{T}^\I =4J+1$.)
\item The elements of $\Sc{T}^0$ are numbered from $1$ to $J$ so that $a_{j_1}\succeq a_{j_2}$ implies $j_1\le j_2$, where $a_j$ is the $j$-th element of $\Sc{T}^0$, which we call the $j$-th parent.
\item For each element of $\Sc{T}^0$ its children are numbered from $1$ to $5$, and the second and the fourth children always have no child.
(The set of all the second and the fourth children, which is then a subset of $\Sc{T}^\I$, is denoted by $\Sc{T}^{\I}_1$, and $\Sc{T}^\I _2:=\Sc{T}^\I \setminus \Sc{T}^\I _1$.
Note that $\# \Sc{T}^\I _1=2J$ and $\# \Sc{T}^\I _2=2J+1$.)
\end{enumerate}
We write $\mathfrak{T}(J)$ to denote the set of all trees of the $J$-th generation.
We also define
\eqq{\mathfrak{I}(J):=\shugo{1,2,\dots ,7}^J,\qquad \mathfrak{U}(J):=\mathfrak{I}(J)\times \mathfrak{T}(J)}
for $J\in \Bo{N}$.
Note that $\# \mathfrak{U}(J)=7^J\prod _{j=1}^J(2j-1)$.
\end{defn} 
\begin{defn}[Index functions]
Given $J\in \Bo{N}$ and $\Sc{T}\in \mathfrak{T}(J)$, an index function $\mathbf{n}=(n_a)_{a\in \Sc{T}}$ on $\Sc{T}$ is a map from $\Sc{T}$ to $\Bo{Z}$ such that
\eqq{n_{a}=n_{a_1}+n_{a_2}+\cdots +n_{a_5}}
for all $a\in \Sc{T}^0$, where $a_l$ stands for the $l$-th child of $a$.
We write $\mathfrak{N}(\Sc{T})$ to denote the set of all index functions on $\Sc{T}$.
Given $\mathbf{n}\in \mathfrak{N}(\Sc{T})$, we define
\eqq{\Phi _j:=n_{a^j}|n_{a^j}|-n_{a_1^j}|n_{a_1^j}|-n_{a_3^j}|n_{a_3^j}|-n_{a_5^j}|n_{a_5^j}|}
for $j=1,2,\cdots ,J$, where $a_j$ means the $j$-th parent.
\end{defn}

Using these notations the equation \eqref{eq_om} can be written as
\eq{eq_om1}{\p _t\om (t,n)&=\!\!\sum _{(\mathbf{i},\Sc{T})\in \mathfrak{U}(1)}\sum _{\mat{\mathbf{n}\in \mathfrak{N}(\Sc{T})\\ n_{\text{root}}=n}}e^{it\Phi _1}\hat{m}_{i_1}\big( n,(n_l)_{l=1}^5\big) \prod _{a\in \Sc{T}^\I _1}\!\!W_a(t,n_a)\prod _{b\in \Sc{T}^\I _2}\!\!\om _b(t,n_b)+\Sc{R}^{(0)}[u,\om ](t,n)\\
:\!\!&=\Sc{N}^{(1)}[u,\om ](t,n)+\Sc{R}^{(0)}[u,\om ](t,n),\\
\p _t\om ^*(t,n)&=\!\!\sum _{(\mathbf{i},\Sc{T})\in \mathfrak{U}(1)}\sum _{\mat{\mathbf{n}\in \mathfrak{N}(\Sc{T})\\ n_{\text{root}}=n}}e^{it\Phi _1}\hat{m}_{i_1}^*\big( n,(n_l)_{l=1}^5\big) \prod _{a\in \Sc{T}^\I _1}\!\!W_a(t,n_a)\prod _{b\in \Sc{T}^\I _2}\!\!\om _b(t,n_b)+\bbar{\Sc{R}^{(0)}[u,\om ](t,-n)}\\
&=\bbar{\Sc{N}^{(1)}[u,\om ](t,-n)}+\bbar{\Sc{R}^{(0)}[u,\om ](t,-n)},}
where $n_l$ denotes the value of $\mathbf{n}$ at the $l$-th child of the root in $\Sc{T}\in \mathfrak{T}(1)$, $W_a$ matches one of $\F [e^{ikF[u(t)]}]$, $k\in \shugo{\pm 1,\pm 3}$, and $\om _b$ is either $\om$ or $\om ^*$.

\smallskip
\subsection{Normal form reduction.}
In this subsection we carry out all computations formally and postpone the justification of each step until Subsection~\ref{sec:proof}.

Now, we proceed to the first normal form reduction step.
Let $M>1$ be a large constant to be chosen at the end of the proof of a priori estimates according to the size of the initial datum.
We first divide the summation over $\mathbf{n}$ in the first equation of \eqref{eq_om1} into resonant and non-resonant frequencies: 
\eqq{\Sc{N}^{(1)}[u,\om ](t,n)&=\sum _{(\mathbf{i},\Sc{T})\in \mathfrak{U}(1)}\sum _{\mat{\mathbf{n}\in \mathfrak{N}(\Sc{T}),\,n_{\text{root}}=n\\ |\Phi _1|\le M}}e^{it\Phi _1}\hat{m}_{i_1}\big( n,(n_l)_{l=1}^5\big) \prod _{a\in \Sc{T}^\I _1}W_a(t,n_a)\prod _{b\in \Sc{T}^\I _2}\om _b(t,n_b)\\
&\hx +\sum _{(\mathbf{i},\Sc{T})\in \mathfrak{U}(1)}\sum _{\mat{\mathbf{n}\in \mathfrak{N}(\Sc{T}),\,n_{\text{root}}=n\\ |\Phi _1|>M}}e^{it\Phi _1}\hat{m}_{i_1}\big( n,(n_l)_{l=1}^5\big) \prod _{a\in \Sc{T}^\I _1}W_a(t,n_a)\prod _{b\in \Sc{T}^\I _2}\om _b(t,n_b)\\
&=:\Sc{N}^{(1)}_R[u,\om ](t,n)+\Sc{N}^{(1)}_{N\!R}[u,\om ](t,n).}
We then apply differentiation by parts with respect to $t$ only to the non-resonant part:
\eqq{&\Sc{N}^{(1)}_{N\!R}[u,\om ](t,n)\\
&=\p _t\Big[ \sum _{(\mathbf{i},\Sc{T})\in \mathfrak{U}(1)}\sum _{\mat{\mathbf{n}\in \mathfrak{N}(\Sc{T}),\,n_{\text{root}}=n\\ |\Phi _1|>M}}\frac{e^{it\Phi _1}\hat{m}_{i_1}\big( n,(n_l)_{l=1}^5\big)}{i\Phi _1}\prod _{a\in \Sc{T}^\I _1}W_a(t,n_a)\prod _{b\in \Sc{T}^\I _2}\om _b(t,n_b)\Big] \\
&\hx -\sum _{(\mathbf{i},\Sc{T})\in \mathfrak{U}(1)}\sum _{\mat{\mathbf{n}\in \mathfrak{N}(\Sc{T}),\,n_{\text{root}}=n\\ |\Phi _1|>M}}\frac{e^{it\Phi _1}\hat{m}_{i_1}\big( n,(n_l)_{l=1}^5\big)}{i\Phi _1}\p _t\Big[ \prod _{a\in \Sc{T}^\I _1}W_a(t,n_a)\Big] \prod _{b\in \Sc{T}^\I _2}\om _b(t,n_b)\\
&\hx -\sum _{(\mathbf{i},\Sc{T})\in \mathfrak{U}(1)}\sum _{\mat{\mathbf{n}\in \mathfrak{N}(\Sc{T}),\,n_{\text{root}}=n\\ |\Phi _1|>M}}\frac{e^{it\Phi _1}\hat{m}_{i_1}\big( n,(n_l)_{l=1}^5\big)}{i\Phi _1}\prod _{a\in \Sc{T}^\I _1}W_a(t,n_a)\p _t\Big[ \prod _{b\in \Sc{T}^\I _2}\om _b(t,n_b)\Big] \\
&=:\p _t\Sc{N}^{(1)}_0[u,\om ](t,n)+\Sc{N}^{(1)}_1[u,\om ](t,n)+\Sc{N}^{(1)}_2[u,\om ](t,n).}
Substituting the equation \eqref{eq_om1} into $\Sc{N}^{(1)}_2[u,\om ]$, we have
\eqq{&\Sc{N}^{(1)}_2[u,\om ](t,n)\\
&=-\sum _{(\mathbf{i},\Sc{T})\in \mathfrak{U}(1)}\sum _{\mat{\mathbf{n}\in \mathfrak{N}(\Sc{T})\\ n_{\text{root}}=n\\ |\Phi _1|>M}}\frac{e^{it\Phi _1}\hat{m}_{i_1}\big( n,(n_l)_{l=1}^5\big)}{i\Phi _1}\prod _{a\in \Sc{T}^\I _1}W_a(t,n_a)\sum _{b\in \Sc{T}^\I _2}\Sc{R}^{(0)}[u,\om ](t,n_b)\prod _{\mat{b'\in \Sc{T}^\I _2\\ b'\neq b}}\om _{b'}(t,n_{b'})\\[-10pt]
&\hx -\sum _{(\mathbf{i},\Sc{T})\in \mathfrak{U}(2)}\sum _{\mat{\mathbf{n}\in \mathfrak{N}(\Sc{T}),\,n_{\text{root}}=n\\ |\Phi _1|>M}}\frac{e^{it(\Phi _1+\Phi _2)}\prod\limits _{j=1}^2\hat{m}_{i_j}\big( n_{a^j},(n_{a^j_l})_{l=1}^5\big)}{i\Phi _1}\prod _{a\in \Sc{T}^\I _1}W_a(t,n_a)\prod _{b\in \Sc{T}^\I _2}\om _b(t,n_b)\\
&=:\Sc{R}^{(1)}[u,\om ](t,n)+\Sc{N}^{(2)}[u,\om ](t,n),}
where $a^j$ and $a^j_l$ ($1\le l\le 5$) denote the $j$-th parent in $\Sc{T}$ and its children, respectively.
It should be remarked that in the above expression, $\hat{m}_{i_2}$ and $\Sc{R}^{(0)}[u,\om ](t,n_b)$ may be replaced with $\hat{m}^*_{i_2}$ and $\bbar{\Sc{R}^{(0)}[u,\om ](t,-n_b)}$, respectively; this is, for instance, in the case where $(\mathbf{i},\Sc{T})\in \mathfrak{U}(2)$ satisfies $i_1=1$ and $a^2=a^1_5$.
However, we may neglect such a difference since it plays no role in the multilinear estimates to be established in the next section.

To summarize, we have obtained the equation
\eq{eq_om2}{\p _t\om =\p _t\Sc{N}^{(1)}_0[u,\om ]+\Sc{N}^{(1)}_R[u,\om ]+\Sc{N}^{(1)}_1[u,\om ]+\sum _{j=0}^1\Sc{R}^{(j)}[u,\om ]+\Sc{N}^{(2)}[u,\om ].}

\smallskip
Let us apply the normal form reduction once more to the term $\Sc{N}^{(2)}[u,\om ]$ in \eqref{eq_om2}, for which the frequencies $\mathbf{n}\in \mathfrak{N}(\Sc{T})$ ($\Sc{T}\in \mathfrak{T}(2)$) are already restricted to the region where $|\Phi _1|>M$ holds.
This time we consider the frequencies satisfying
\eqq{|\Phi _1|>M,\quad |\Phi _1+\Phi _2|\le 2|\Phi _1|\qquad \text{or}\qquad |\Phi _1|>M,\quad |\Phi _1+\Phi _2|>2|\Phi _1|}
as resonant or non-resonant frequencies, respectively.
It holds $|\Phi _1|^{-1}\lec (\LR{\Phi _1}\LR{\Phi _2})^{-1/2}$ in the resonant case, whereas we have $|\Phi _1+\Phi _2|>2M$ and $|\Phi _1+\Phi _2|^{-1}\sim \LR{\Phi _2}^{-1}$ in the non-resonant case.
Similarly to the first reduction, we divide $\Sc{N}^{(2)}[u,\om ]$ as
\eqq{\Sc{N}^{(2)}[u,\om ](n)&=c\sum _{(\mathbf{i},\Sc{T})\in \mathfrak{U}(2)}\bigg[ \sum _{\mat{\mathbf{n}\in \mathfrak{N}(\Sc{T}),\,n_{\text{root}}=n\\ |\Phi _1|>M\\ |\Phi _1+\Phi _2|\le 2|\Phi _1|}}+\sum _{\mat{\mathbf{n}\in \mathfrak{N}(\Sc{T}),\,n_{\text{root}}=n\\ |\Phi _1|>M\\ |\Phi _1+\Phi _2|>2|\Phi _1|}}\bigg] \frac{e^{it(\Phi _1+\Phi _2)}\prod\limits _{j=1}^2\hat{m}_{i_j}\big( n_{a^j},(n_{a^j_l})_{l=1}^5\big)}{\Phi _1}\\[-10pt]
&\hspace{230pt}\times \prod _{a\in \Sc{T}^\I _1}W_a(n_a)\prod _{b\in \Sc{T}^\I _2}\om _b(n_b)\\[-10pt]
&=:\Sc{N}^{(2)}_R[u,\om ](n)+\Sc{N}^{(2)}_{N\!R}[u,\om ](n),}
(where $c$ denotes a complex constant with $|c|=1$, which plays no role in the estimates,) and differentiate by parts as
\eqq{&\Sc{N}^{(2)}_{N\!R}[u,\om ](n)\\
&=\p _t\bigg[ c\sum _{(\mathbf{i},\Sc{T})\in \mathfrak{U}(2)}\sum _{\mat{\mathbf{n}\in \mathfrak{N}(\Sc{T}),\,n_{\text{root}}=n\\ |\Phi _1|>M\\ |\Phi _1+\Phi _2|>2|\Phi _1|}}\frac{e^{it(\Phi _1+\Phi _2)}\prod\limits _{j=1}^2\hat{m}_{i_j}\big( n_{a^j},(n_{a^j_l})_{l=1}^5\big)}{\Phi _1(\Phi _1+\Phi _2)}\prod _{a\in \Sc{T}^\I _1}W_a(n_a)\prod _{b\in \Sc{T}^\I _2}\om _b(n_b)\bigg] \\
&\hx +c\sum _{(\mathbf{i},\Sc{T})}\sum _{\mathbf{n}}\frac{e^{it(\Phi _1+\Phi _2)}\prod _j\hat{m}_{i_j}\big( n_{a^j},(n_{a^j_l})_{l=1}^5\big)}{\Phi _1(\Phi _1+\Phi _2)}\p _t\bigg[ \prod _{a\in \Sc{T}^\I _1}W_a(n_a)\bigg] \prod _{b\in \Sc{T}^\I _2}\om _b(n_b)\\
&\hx +c\sum _{(\mathbf{i},\Sc{T})}\sum _{\mathbf{n}}\frac{e^{it(\Phi _1+\Phi _2)}\prod _j\hat{m}_{i_j}\big( n_{a^j},(n_{a^j_l})_{l=1}^5\big)}{\Phi _1(\Phi _1+\Phi _2)}\prod _{a\in \Sc{T}^\I _1}W_a(n_a)\p _t\bigg[ \prod _{b\in \Sc{T}^\I _2}\om _b(n_b)\bigg] \\
&=:\p _t\Sc{N}^{(2)}_0[u,\om ](n)+\Sc{N}^{(2)}_1[u,\om ](n)+\Sc{N}^{(2)}_2[u,\om ](n),}
and then substitute the equation \eqref{eq_om1} to obtain
\eqq{&\Sc{N}^{(2)}_2[u,\om ](n)\\[-15pt]
&=c\sum _{(\mathbf{i},\Sc{T})\in \mathfrak{U}(2)}\sum _{\mat{\mathbf{n}\in \mathfrak{N}(\Sc{T}),\,n_{\text{root}}=n\\ |\Phi _1|>M\\ |\Phi _1+\Phi _2|>2|\Phi _1|}}\frac{e^{it(\Phi _1+\Phi _2)}\prod\limits _{j=1}^2\hat{m}_{i_j}\big( n_{a^j},(n_{a^j_l})_{l=1}^5\big)}{\Phi _1(\Phi _1+\Phi _2)}\prod _{a\in \Sc{T}^\I _1}W_a(n_a)\\[-15pt]
&\hspace{240pt}\times \sum _{b\in \Sc{T}^\I _2}\Sc{R}^{(0)}[u,\om ](n_b)\prod _{\mat{b'\in \Sc{T}^\I _2\\ b'\neq b}}\om _{b'}(n_{b'})\\[-10pt]
&\hx +c\sum _{(\mathbf{i},\Sc{T})\in \mathfrak{U}(3)}\sum _{\mat{\mathbf{n}\in \mathfrak{N}(\Sc{T}),\,n_{\text{root}}=n\\ |\Phi _1|>M\\ |\Phi _1+\Phi _2|>2|\Phi _1|}}\frac{e^{it(\Phi _1+\Phi _2+\Phi _3)}\prod\limits _{j=1}^3\hat{m}_{i_j}\big( n_{a^j},(n_{a^j_l})_{l=1}^5\big)}{\Phi _1(\Phi _1+\Phi _2)}\prod _{a\in \Sc{T}^\I _1}W_a(n_a)\prod _{b\in \Sc{T}^\I _2}\om _b(n_b)\\
&=:\Sc{R}^{(2)}[u,\om ](n)+\Sc{N}^{(3)}[u,\om ](n).}
We again neglect the difference between $\hat{m}_i$ and $\hat{m}_i^*$ or $\Sc{R}^{(0)}[u,\om ](n_b)$ and $\bbar{\Sc{R}^{(0)}[u,\om ](-n_b)}$, which plays no role in the estimates.
As a consequence, we have the equation
\eqq{\p _t\om =\sum _{j=1}^2\bigg( \p _t\Sc{N}^{(j)}_0[u,\om ]+\Sc{N}^{(j)}_R[u,\om ]+\Sc{N}^{(j)}_1[u,\om ]\bigg) +\sum _{j=0}^2\Sc{R}^{(j)}[u,\om ]+\Sc{N}^{(3)}[u,\om ]}
after the second normal form reduction.

To describe the general step, we define 
\eqq{\Sc{M}_R^{(J)}&:=\begin{cases}
\{ \mu \in \Bo{Z}:|\mu |\le M\} &\text{if $J=1$},\\
\{ (\mu _1,\mu _2)\in \Bo{Z}^2:|\mu _1|>M,\,|\ti{\mu}_2|\le 2|\mu _1|\} &\text{if $J=2$},\\
\{ (\mu _j)_{j=1}^J\in \Bo{Z}^J:|\mu _1|>M,\,|\ti{\mu}_j|>2|\ti{\mu}_{j-1}|~(2\le j\le J\!-\!1),\,|\ti{\mu}_J|\le 2|\ti{\mu}_{J-1}|\} &\text{if $J\ge 3$},
\end{cases}\\
\Sc{M}_{N\!R}^{(J)}&:=\begin{cases}
\{ \mu \in \Bo{Z}:|\mu |>M\} &\text{if $J=1$},\\
\{ (\mu _j)_{j=1}^J\in \Bo{Z}^J:|\mu _1|>M,\,|\ti{\mu}_j|>2|\ti{\mu}_{j-1}|~(2\le j\le J)\}\hspace{91pt} &\text{if $J\ge 2$},
\end{cases}}
where $\ti{\mu}_j:=\mu _1+\mu _2+\cdots +\mu _j$, so that $\Sc{M}^{(J)}_{N\!R}\times \Bo{Z}=\Sc{M}^{(J+1)}_R\cup \Sc{M}^{(J+1)}_{N\!R}$.
Then, the equation for $\om$ obtained after the $J$-th normal form reduction is written as 
\eq{eq_om3}{\p _t\om =\sum _{j=1}^J\bigg( \p _t\Sc{N}^{(j)}_0[u,\om ]+\Sc{N}^{(j)}_R[u,\om ]+\Sc{N}^{(j)}_1[u,\om ]\bigg) +\sum _{j=0}^J\Sc{R}^{(j)}[u,\om ]+\Sc{N}^{(J+1)}[u,\om ],}
where $\ti{\Phi}_j:=\Phi _1+\Phi _2+\cdots +\Phi _j$ and%
\footnote{For $J=1$ we use the convention that $\prod _{j=1}^0p_j=1$ ($p_j\in \R$).}
\eqq{&\Sc{N}_R^{(J)}[u,\om ](n):=c\!\!\!\!\!\sum _{(\mathbf{i},\Sc{T})\in \mathfrak{U}(J)}\sum _{\mat{\mathbf{n}\in \mathfrak{N}(\Sc{T}),\,n_{\text{root}}=n\\ (\Phi _j)_{j=1}^J\in \Sc{M}_R^{(J)}}}\!\!\!\!\!\!\!\!\frac{e^{it\ti{\Phi}_J}\prod\limits _{j=1}^J\hat{m}_{i_j}\big( n_{a^j},(n_{a^j_l})_{l=1}^5\big)}{\prod\limits _{j=1}^{J-1}\ti{\Phi}_j}\prod _{a\in \Sc{T}^\I _1}W_a(n_a)\prod _{b\in \Sc{T}^\I _2}\om _b(n_b),\\
&\Sc{N}_0^{(J)}[u,\om ](n):=c\!\!\!\!\!\sum _{(\mathbf{i},\Sc{T})\in \mathfrak{U}(J)}\sum _{\mat{\mathbf{n}\in \mathfrak{N}(\Sc{T}),\,n_{\text{root}}=n\\ (\Phi _j)_{j=1}^J\in \Sc{M}_{N\!R}^{(J)}}}\!\!\!\!\!\!\!\!\frac{e^{it\ti{\Phi}_J}\prod\limits _{j=1}^J\hat{m}_{i_j}\big( n_{a^j},(n_{a^j_l})_{l=1}^5\big)}{\prod\limits _{j=1}^J\ti{\Phi}_j}\prod _{a\in \Sc{T}^\I _1}W_a(n_a)\prod _{b\in \Sc{T}^\I _2}\om _b(n_b),\\
&\Sc{N}_1^{(J)}[u,\om ](n):=c\!\!\!\!\!\sum _{(\mathbf{i},\Sc{T})\in \mathfrak{U}(J)}\sum _{\mat{\mathbf{n}\in \mathfrak{N}(\Sc{T}),\,n_{\text{root}}=n\\ (\Phi _j)_{j=1}^J\in \Sc{M}_{N\!R}^{(J)}}}\!\!\!\!\!\!\!\!\frac{e^{it\ti{\Phi}_J}\prod\limits _{j=1}^J\hat{m}_{i_j}\big( n_{a^j},(n_{a^j_l})_{l=1}^5\big)}{\prod\limits _{j=1}^J\ti{\Phi}_j}\p _t\bigg[ \prod _{a\in \Sc{T}^\I _1}\!\!W_a(n_a)\bigg] \prod _{b\in \Sc{T}^\I _2}\om _b(n_b),\\
&\Sc{R}^{(J)}[u,\om ](n):=c\!\!\!\!\!\sum _{(\mathbf{i},\Sc{T})\in \mathfrak{U}(J)}\sum _{\mat{\mathbf{n}\in \mathfrak{N}(\Sc{T}),\,n_{\text{root}}=n\\ (\Phi _j)_{j=1}^J\in \Sc{M}_{N\!R}^{(J)}}}\!\!\!\!\!\!\!\!\frac{e^{it\ti{\Phi}_J}\prod\limits _{j=1}^J\hat{m}_{i_j}\big( n_{a^j},(n_{a^j_l})_{l=1}^5\big)}{\prod\limits _{j=1}^J\ti{\Phi}_j}\prod _{a\in \Sc{T}^\I _1}W_a(n_a)\\[-10pt]
&\hspace{260pt}\times \sum _{b\in \Sc{T}^\I _2}\Sc{R}^{(0)}[u,\om ](n_b)\prod _{\mat{b'\in \Sc{T}^\I _2\\ b'\neq b}}\om _{b'}(n_{b'}),\\[-10pt]
&\Sc{N}^{(J+1)}[u,\om ](n):=c\!\!\!\!\!\sum _{(\mathbf{i},\Sc{T})\in \mathfrak{U}(J+1)}\sum _{\mat{\mathbf{n}\in \mathfrak{N}(\Sc{T}),\,n_{\text{root}}=n\\ (\Phi _j)_{j=1}^J\in \Sc{M}_{N\!R}^{(J)}}}\!\!\!\!\!\!\!\!\frac{e^{it\ti{\Phi}_{J+1}}\prod\limits _{j=1}^{J+1}\hat{m}_{i_j}\big( n_{a^j},(n_{a^j_l})_{l=1}^5\big)}{\prod\limits _{j=1}^J\ti{\Phi}_j}\!\prod _{a\in \Sc{T}^\I _1}\!\!W_a(n_a)\prod _{b\in \Sc{T}^\I _2}\!\!\om _b(n_b).}
In fact, we make a resonant/non-resonant decomposition of the term $\Sc{N}^{(J)}$ in the equation after  the $(J-1)$-th reduction as
\eqs{\Sc{N}^{(J)}[u,\om ](n)=\Sc{N}_{R}^{(J)}[u,\om ](n)+\Sc{N}_{N\!R}^{(J)}[u,\om ](n),\\
\Sc{N}_{N\!R}^{(J)}[u,\om ](n):=c\!\!\!\!\!\sum _{(\mathbf{i},\Sc{T})\in \mathfrak{U}(J)}\sum _{\mat{\mathbf{n}\in \mathfrak{N}(\Sc{T}),\,n_{\text{root}}=n\\ (\Phi _j)_{j=1}^J\in \Sc{M}_{N\!R}^{(J)}}}\!\!\!\!\!\!\!\!\frac{e^{it\ti{\Phi}_J}\prod\limits _{j=1}^J\hat{m}_{i_j}\big( n_{a^j},(n_{a^j_l})_{l=1}^5\big)}{\prod\limits _{j=1}^{J-1}\ti{\Phi}_j}\prod _{a\in \Sc{T}^\I _1}W_a(n_a)\prod _{b\in \Sc{T}^\I _2}\om _b(n_b),
}
and apply differentiation by parts to $\Sc{N}_{N\!R}^{(J)}[u,\om ](n)$ as
\eqs{\Sc{N}_{N\!R}^{(J)}[u,\om ](n)=\p _t\Sc{N}_0^{(J)}[u,\om ](n)+\Sc{N}_1^{(J)}[u,\om ](n)+\Sc{N}_2^{(J)}[u,\om ](n),\\
\Sc{N}_2^{(J)}[u,\om ](n):=c\!\!\!\!\!\sum _{(\mathbf{i},\Sc{T})\in \mathfrak{U}(J)}\sum _{\mat{\mathbf{n}\in \mathfrak{N}(\Sc{T}),\,n_{\text{root}}=n\\ (\Phi _j)_{j=1}^J\in \Sc{M}_{N\!R}^{(J)}}}\!\!\!\!\!\!\!\!\!\!\!\frac{e^{it\ti{\Phi}_J}\prod\limits _{j=1}^J\hat{m}_{i_j}\big( n_{a^j},(n_{a^j_l})_{l=1}^5\big)}{\prod\limits _{j=1}^J\ti{\Phi}_j}\!\!\!\prod _{a\in \Sc{T}^\I _1}\!\!W_a(n_a)\cdot \p _t\bigg[ \prod _{b\in \Sc{T}^\I _2}\!\!\om _b(n_b)\bigg] ,
}
and substitute the equation \eqref{eq_om1}:
\eqq{\Sc{N}_2^{(J)}[u,\om ](n)=\Sc{R}^{(J)}[u,\om ](n)+\Sc{N}^{(J+1)}[u,\om ](n),}
which leads to \eqref{eq_om3}.

Integrating \eqref{eq_om3} with respect to $t$, we obtain 
\eq{eq_om4}{\om (\cdot ,n)\Big| _0^t&=\sum _{j=1}^J\Sc{N}^{(j)}_0[u,\om ](\cdot ,n)\Big| _0^t+\sum _{j=1}^J\int _0^t\Big[ \Sc{N}^{(j)}_R[u,\om ](\tau ,n)+\Sc{N}^{(j)}_1[u,\om ](\tau ,n)\Big] \,d\tau \\
&\hx +\sum _{j=0}^J\int _0^t\Sc{R}^{(j)}[u,\om ] (\tau ,n)\,d\tau +\int _0^t\Sc{N}^{(J+1)}[u,\om ](\tau ,n)\,d\tau .}
This integral form will be used to derive a priori bounds on $\om$ in Subsection~\ref{sec:proof}.

As observed in \cite[Section~2.4]{K-all}, we see the following:
\begin{itemize}
\item If $\Phi _1=\ti{\Phi}_1\in \Sc{M}_R^{(1)}$, then $\big| e^{it\ti{\Phi}_1}\hat{m}_{i_1}\big( n,\,(n_l)_{l=1}^5\big) \big| \lec M^{1/2}\LR{\Phi _1}^{-1/2}\big| \hat{m}_{i_1}\big( n,\,(n_l)_{l=1}^5\big) \big|$.
\item If $J\ge 2$ and $(\Phi _j)_{j=1}^J\in \Sc{M}_R^{(J)}$, then $|\Phi _j|\sim |\ti{\Phi}_j|>2^{j-1}M$ ($1\le j\le J-1$), $|\Phi _J|\lec |\ti{\Phi}_{J-1}|$ and
\eq{est:Phi1}{\left| \frac{e^{it\ti{\Phi}_J}\prod_{j=1}^J\hat{m}_{i_j}\big( n_{a^j},(n_{a^j_l})_{l=1}^5\big)}{\prod _{j=1}^{J-1}\ti{\Phi}_j}\right| \le C^J\frac{\prod_{j=1}^J\left| \hat{m}_{i_j}\big( n_{a^j},(n_{a^j_l})_{l=1}^5\big) \right|}{\prod _{j=1}^{J-2}(2^{j-1}M)^{1/2}\prod _{j=1}^{J}\LR{\Phi _j}^{1/2}}.}
\item If $J\ge 1$ and $(\Phi _j)_{j=1}^J\in \Sc{M}_{N\!R}^{(J)}$, then $|\Phi _j|\sim |\ti{\Phi}_j|>2^{j-1}M$ ($1\le j\le J$) and
\begin{align}
\label{est:Phi2} 
\left| \frac{e^{it\ti{\Phi}_J}\prod_{j=1}^J\hat{m}_{i_j}\big( n_{a^j},(n_{a^j_l})_{l=1}^5\big)}{\prod _{j=1}^{J}\ti{\Phi}_j}\right| &\le C^J\frac{\prod_{j=1}^J\left| \hat{m}_{i_j}\big( n_{a^j},(n_{a^j_l})_{l=1}^5\big) \right|}{\prod _{j=1}^{J}(2^{j-1}M)^{\de}\prod _{j=1}^{J}\LR{\Phi _j}^{1-\de}}\\
\label{est:Phi2'}
&\le C^J\frac{\prod_{j=1}^J\left| \hat{m}_{i_j}\big( n_{a^j},(n_{a^j_l})_{l=1}^5\big) \right|}{\prod _{j=1}^{J}(2^{j-1}M)^{1/2}\prod _{j=1}^{J}\LR{\Phi _j}^{1/2}}
\end{align}
for any $\de \in (0,1/2]$.
\end{itemize}
These inequalities allow us to reduce the multilinear estimates of degree $4J+1$ for the terms appearing after the $J$-th normal form reduction to just $J$ repetitions of fundamental quintilinear estimates given in Proposition~\ref{prop:fundamental} below.

Finally, we recall that the above computations in each step of the normal form reduction, which include switch of the order of summation and time differentiation, application of the product rule for time differentiation, and substitution of the equation \eqref{eq_om1}, may not be justified for general distributional solutions $u\in C_TH^s$, $\om \in C_T\ell ^2_s$ if $s$ is not large enough.
This issue will be handled in Subsection~\ref{sec:proof} by using multilinear estimates to be established in Subsection~\ref{sec:multi}.


%

\smallskip
\subsection{Multilinear estimates}\label{sec:multi}

We begin with stating the fundamental estimates for quintilinear forms, which will be proved in Section~\ref{sec:proof-multi}.
\begin{prop}\label{prop:fundamental}
Let $s>1/2$, and let $\hat{m}$ (\mbox{resp.} $\hat{\hat{m}}$) be one of $\shugo{\hat{m}_i,\hat{m}_i^*\,;\,1\le i\le 7}$ (\mbox{resp.} $\{\hat{\hat{m}}_i\,;\,1\le i\le 7\}$).
There exists $\de >0$ such that we have
\begin{align}
\begin{split}
&\norm{\sum _{n=n_{12345}}\big| \hat{\hat{m}}\big| \;\om _1(n_1)W_2(n_2)\om _3(n_3)W_4(n_4)\om _5(n_5)}{\ell ^2_s} \\[-5pt]
&\hspace{90pt} \lec \norm{\om _1}{\ell ^2_s}\norm{\om _3}{\ell ^2_s}\norm{\om _5}{\ell ^2_s}\Big( \norm{W_2}{\ell ^2_{s+1}}\norm{W_4}{\ell ^2_{s}}+\norm{W_2}{\ell ^2_{s}}\norm{W_4}{\ell ^2_{s+1}}\Big) ,
\end{split}\label{est:matome-0}\\[10pt]
\begin{split}
&\norm{\sum _{n=n_{12345}}\frac{\big| \hat{m}\big|}{\LR{\Phi}^{1/2}}\om _1(n_1)W_2(n_2)\om _3(n_3)W_4(n_4)\om _5(n_5)}{\ell ^2_s} \\
&\hspace{200pt} \lec \norm{\om _1}{\ell ^2_s}\norm{W_2}{\ell ^2_s}\norm{\om _3}{\ell ^2_s}\norm{W_4}{\ell ^2_s}\norm{\om _5}{\ell ^2_s},
\end{split}\label{est:matome-1}\\[10pt]
\begin{split}
&\norm{\sum _{n=n_{12345}}\frac{\big| \hat{m}\big|}{\LR{\Phi}}\om _1(n_1)W_2(n_2)\om _3(n_3)W_4(n_4)\om _5(n_5)}{\ell ^2_s} \\[-5pt]
&\hspace{90pt} \lec \norm{\om _1}{\ell ^2_s}\norm{\om _3}{\ell ^2_s}\norm{\om _5}{\ell ^2_s}\Big( \norm{W_2}{\ell ^2_{s-1}}\norm{W_4}{\ell ^2_{s}}\wedge \norm{W_2}{\ell ^2_{s}}\norm{W_4}{\ell ^2_{s-1}}\Big) ,
\end{split}\label{est:matome-2}\\[10pt]
\begin{split}
&\norm{\sum _{n=n_{12345}}\frac{\big| \hat{m}\big|}{\LR{\Phi}^{1-\de}}\om _1(n_1)W_2(n_2)\om _3(n_3)W_4(n_4)\om _5(n_5)}{\ell ^2_{s-1}} \\
&\lec \Big( \norm{\om _1}{\ell ^2_{s-1}}\norm{\om _3}{\ell ^2_s}\norm{\om _5}{\ell ^2_s}\wedge \norm{\om _1}{\ell ^2_s}\norm{\om _3}{\ell ^2_{s-1}}\norm{\om _5}{\ell ^2_s}\wedge \norm{\om _1}{\ell ^2_s}\norm{\om _3}{\ell ^2_s}\norm{\om _5}{\ell ^2_{s-1}}\Big) \norm{W_2}{\ell ^2_s}\norm{W_4}{\ell ^2_s}
\end{split}\label{est:matome-3}
\end{align}
for any real-valued non-negative functions $\om _1,\om _3,\om _5$, $W_2,W_4$.
\end{prop}

Now, we deduce the required estimates of multilinear forms from the quintilinear estimates given in Proposition~\ref{prop:fundamental}.
\begin{prop}\label{prop}
Let $1/2<s\le 1$.

(i) There exists $0<\de <1/2$ depending on $s$ such that the following holds.
Let $u\in C_TH^s$ be a solution of \eqref{mBO'} in the sense of distribution, and define the solution $\om \in C_T\ell ^2_s$ of \eqref{eq_om1} by \eqref{utov} and \eqref{vtoom}.
Then, there exists $C_0>0$ depending on $s$, $\tnorm{u}{C_TH^s}$ (and $\tnorm{\om}{C_T\ell ^2_s}$)%
\footnote{Note that $\norm{\om}{C_T\ell ^2_s}=\norm{v}{C_TH^s}$, which is controlled by $\norm{u}{C_TH^s}$ using Lemma~\ref{lem:exp}.} 
such that for any $J\in \Bo{N}$ we have
\eqq{\norm{\Sc{R}^{(0)}[u,\om ]}{C_T\ell ^2_s}+\norm{\Sc{N}^{(1)}[u,\om ]}{C_T\ell ^2_{s-1}}&\le C_0,\\
\norm{\Sc{N}^{(J)}_R[u,\om ]}{C_T\ell ^2_s}&\le M\big( C_0M^{-1/2}\big) ^J,\\
\norm{\Sc{N}^{(J)}_0[u,\om ]}{C_T\ell ^2_s}&\le \big( C_0M^{-1/2}\big) ^J,\\
\norm{\Sc{N}^{(J)}_1[u,\om ]}{C_T\ell ^2_s}&\le C_0M^{1/2}\big( C_0M^{-1/2}\big) ^J,\\
\norm{\Sc{R}^{(J)}[u,\om ]}{C_T\ell ^2_s}&\le C_0\big( C_0M^{-1/2}\big) ^J,\\
\norm{\Sc{N}^{(J+1)}[u,\om ]}{C_T\ell ^2_{s-1}}&\le C_0\big( C_0M^{-\de}\big) ^J.
}

(ii) Moreover, if $\ti{u}\in C_TH^s$ is another solution of \eqref{mBO'} and $\ti{\om}\in C_T\ell ^2_s$ is the corresponding solution of \eqref{eq_om1}, then there exists $\ti{C}_0>0$ depending on $s$, $\tnorm{u}{C_TH^s}$, $\tnorm{\ti{u}}{C_TH^s}$ (and $\tnorm{\om}{C_T\ell ^2_s}$, $\tnorm{\ti{\om}}{C_T\ell ^2_s}$) such that we have
\eqq{\norm{\Sc{R}^{(0)}[u,\om ]-\Sc{R}^{(0)}[\ti{u},\ti{\om}]}{C_T\ell ^2_s}&\le \ti{C}_0\big[
\norm{u-\ti{u}}{C_TH^s}+\norm{\om -\ti{\om}}{C_T\ell ^2_s}\big] ,\\
\norm{\Sc{N}^{(J)}_R[u,\om ]-\Sc{N}^{(J)}_R[\ti{u},\ti{\om}]}{C_T\ell ^2_s}&\le M\big( \ti{C}_0M^{-1/2}\big) ^J\big[
\norm{u-\ti{u}}{C_TH^s}+\norm{\om -\ti{\om}}{C_T\ell ^2_s}\big] ,\\
\norm{\Sc{N}^{(J)}_0[u,\om ]\!-\!\Sc{N}^{(J)}_0[\ti{u},\ti{\om}]}{C_T\ell ^2_s}&\le \big( \ti{C}_0M^{-1/2}\big) ^J\big[
\norm{u-\ti{u}}{C_TH^s}+\norm{\om -\ti{\om}}{C_T\ell ^2_s}\big] ,\\
\norm{\Sc{N}^{(J)}_1[u,\om ]-\Sc{N}^{(J)}_1[\ti{u},\ti{\om}]}{C_T\ell ^2_s}&\le \ti{C}_0M^{1/2}\big( \ti{C}_0M^{-1/2}\big) ^J\big[
\norm{u-\ti{u}}{C_TH^s}+\norm{\om -\ti{\om}}{C_T\ell ^2_s}\big] ,\\
\norm{\Sc{R}^{(J)}[u,\om ]\!-\!\Sc{R}^{(J)}[\ti{u},\ti{\om}]}{C_T\ell ^2_s}&\le \ti{C}_0\big( \ti{C}_0M^{-1/2}\big) ^J\big[
\norm{u-\ti{u}}{C_TH^s}+\norm{\om -\ti{\om}}{C_T\ell ^2_s}\big] .
}
\end{prop}

\begin{proof}
(i) 
\emph{Estimate on $\Sc{R}^{(0)}$}.
We have
\eqq{\norm{\Sc{R}^{(0)}[u,\om ]}{\ell ^2_s}\le \sum _{i=1}^7\norm{\Sc{Q}\big( e^{it\Phi}\hat{\hat{m}}_i;\;u(t),\om (t)\big)}{\ell ^2_s}+\norm{R[u]}{H^s}.}
By \eqref{est:matome-0}, the first term in the right-hand side is bounded by $X_s(u)X_{s+1}(u)\tnorm{\om}{\ell ^2_s}^3$, which is controlled in terms of $\norm{u}{C_TH^s}$ by Lemma~\ref{lem:exp}.
The second term has been estimated in \eqref{est:Ru}.

\emph{Estimate on $\Sc{N}^{(1)}$}.
The product estimate \eqref{prod-est1} easily implies that
\eq{est:N1}{\norm{\sum _{n=n_{12345}}|\hat{m}|\om _1(n_1)W_2(n_2)\om _3(n_3)W_4(n_4)\om _5(n_5)}{\ell ^2_{s-1}}\lec \norm{\om _1}{\ell ^2_s}\norm{W_2}{\ell ^2_s}\norm{\om _3}{\ell ^2_s}\norm{W_4}{\ell ^2_s}\norm{\om _5}{\ell ^2_s},}
where $\hat{m}$ is one of $\shugo{\hat{m}_i,\hat{m}_i^*;\,1\le i\le 7}$.
This and Lemma~\ref{lem:exp} verifies the estimate on  $\Sc{N}^{(1)}$.

\emph{Estimate on $\Sc{N}_R^{(J)}$}.
By invoking \eqref{est:Phi1} and applying \eqref{est:matome-1} $J$ times, we have
\eqq{&\norm{\Sc{N}^{(J)}_R[u,\om ]}{\ell ^2_s}\\
&\le C^J\sum _{(\mathbf{i},\Sc{T})\in \mathfrak{U}(J)}\bigg\| \sum _{\mat{\mathbf{n}\in \mathfrak{N}(\Sc{T})\\ n_{\text{root}}=n}}\frac{\prod _{j=1}^J\big| \hat{m}_{i_j}\big( n_{a^j},(n_{a^j_l})_{l=1}^5\big) \big|}{\prod _{j=1}^{J-2}(2^{j-1}M)^{1/2}\prod _{j=1}^{J}\LR{\Phi _j}^{1/2}}\prod _{a\in \Sc{T}^\I _1}\big| W_a(n_a)\big| \prod _{b\in \Sc{T}^\I _2}\big| \om _b(n_b)\big| \bigg\| _{\ell ^2_s}\\
&\le C^JM^{-(J-2)/2}\sup _{(\mathbf{i},\Sc{T})\in \mathfrak{U}(J)}\bigg\| \sum _{\mat{\mathbf{n}\in \mathfrak{N}(\Sc{T})\\ n_{\text{root}}=n}}\prod _{j=1}^J\frac{\big| \hat{m}_{i_j}\big( n_{a^j},(n_{a^j_l})_{l=1}^5\big) \big|}{\LR{\Phi _j}^{1/2}}\prod _{a\in \Sc{T}^\I _1}\big| W_a(n_a)\big| \prod _{b\in \Sc{T}^\I _2}\big| \om _b(n_b)\big| \bigg\| _{\ell ^2_s}\\
&\le C^JM^{-(J-2)/2}X_s(u)^{2J}\tnorm{\om}{\ell ^2_s}^{2J+1}
}
for $J\ge 2$, where we have used the fact that for $p>1$ there exists $C>0$ such that $\# \mathfrak{U}(J)=7^J\prod _{j=1}^J(2j-1)\le C^J\prod _{j=1}^{J-2}p^{j-1}$, $J\ge 2$.
A similar argument gives the estimate for $J=1$.
Then, Lemma~\ref{lem:exp} implies the desired estimate.

\emph{Estimate on $\Sc{N}_0^{(J)}$} follows in the same manner, by use of \eqref{est:Phi2'} instead of \eqref{est:Phi1}.

\emph{Estimate on $\Sc{N}_1^{(J)}$}.
We note that
\eqq{\p _t\bigg[ \prod _{a\in \Sc{T}^\I _1}W_a\bigg] =\sum _{j_*=1}^J\Big( W_{a^{j_*}_2}\p _tW_{a^{j_*}_4}+W_{a^{j_*}_4}\p _tW_{a^{j_*}_2}\Big) \prod _{\mat{j=1\\ j\neq j_*}}^JW_{a^{j}_2}W_{a^{j}_4}.}
Similarly to \eqref{est:Phi2'}, for each $j_*$ we have
\eqq{\left| \frac{e^{it\ti{\Phi}_J}\prod_{j=1}^J\hat{m}_{i_j}\big( n_{a^j},(n_{a^j_l})_{l=1}^5\big)}{\prod _{j=1}^{J}\ti{\Phi}_j}\right| \lec \frac{\prod_{j=1}^J\big| \hat{m}_{i_j}\big( n_{a^j},(n_{a^j_l})_{l=1}^5\big) \big|}{\prod _{j=1}^{J-1}(2^{j-1}M)^{1/2}\LR{\Phi _{j_*}}\prod _{j=1,\,j\neq j_*}^{J}\LR{\Phi _j}^{1/2}}.}
By these facts and \eqref{est:matome-1} as well as \eqref{est:matome-2}, we have
\eqq{&\norm{\Sc{N}^{(J)}_1[u,\om ]}{\ell ^2_s}\le C^JM^{-(J-1)/2}\max _{k\in \shugo{\pm 1,\pm 3}}\norm{\p _te^{ikF[u]}}{H^{s-1}}X_s(u)^{2J-1}\tnorm{\om}{\ell ^2_s}^{2J+1}.}
Since $u\in C_TH^s$ satisfies
\eq{eq:W}{\p _te^{ikF[u]}=ike^{ikF[u]}\Big( -2P_{\neq c}(u\H \p _xu)-2iB(u,u)+\sgm P_{\neq c}\big[ \big( P_{\neq c}(u^2)\big) ^2\big] \Big)}
(which is verified in the sense of distribution or in $C_TH^{s-1}$), by \eqref{prod-est1} and \eqref{prod-est2} we have
\eqq{\max _{k\in \shugo{\pm 1,\pm 3}}\norm{\p _te^{ikF[u]}}{H^{s-1}}&\lec X_s(u)\LR{\tnorm{u}{H^s}}^4.}
Hence, the claim follows from Lemma~\ref{lem:exp}.

\emph{Estimate on $\Sc{R}^{(J)}$}.
From the estimate on $\Sc{N}^{(J)}_0$, we have
\eqq{\norm{\Sc{R}^{(J)}[u,\om ]}{\ell ^2_s}&\le C^JM^{-J/2}X_s(u)^{2J}\tnorm{\om}{\ell ^2_s}^{2J}\norm{\Sc{R}^{(0)}[u,\om ]}{\ell ^2_s}.}
The claim then follows from the estimate of $\Sc{R}^{(0)}$ shown above and Lemma~\ref{lem:exp}.

\emph{Estimate on $\Sc{N}^{(J+1)}$}.
The estimates \eqref{est:matome-3}, \eqref{est:matome-1} combined with \eqref{est:Phi2} and \eqref{est:N1} yield that
\eqq{\norm{\Sc{N}^{(J+1)}[u,\om ]}{\ell ^2_{s-1}}&\le C^{J+1}M^{-J\de}X_s(u)^{2J+2}\tnorm{\om}{\ell ^2_s}^{2J+3}.}
In the above estimates, for each $\Sc{T}\in \mathfrak{T}(J+1)$ we have applied \eqref{est:matome-3} and \eqref{est:matome-1} $J$ times in total, and the number of application of \eqref{est:matome-3} is equal to the number of elements $a\in \Sc{T}$ such that $a^1\succeq a\succeq a^{J+1}$ and $a\neq a^{J+1}$.
The desired estimate is obtained by Lemma~\ref{lem:exp}.

(ii) 
A similar argument verifies the difference estimates.
In fact, we may just replace $X_s(u)$, $X_{s+1}(u)$ and $\norm{\om}{\ell ^2_s}$ by $X_s(u)+X_s(\ti{u})$, $X_{s+1}(u)+X_{s+1}(\ti{u})$ and $\norm{\om}{\ell ^2_s}+\norm{\ti{\om}}{\ell ^2_s}$, with one exception which is replaced by $Y_s(u,\ti{u})$, $Y_{s+1}(u,\ti{u})$ and $\norm{\om -\ti{\om}}{\ell ^2_s}$, respectively.
For the estimate on the difference of $\Sc{N}^{(J)}_1$, we note that
\eqq{&\max _{k\in \shugo{\pm 1,\pm 3}}\norm{\p _t\big( e^{ikF[u]}-e^{ikF[\ti{u}]}\big)}{H^{s-1}}\\
&\hx \lec Y_s(u,\ti{u})\LR{\tnorm{u}{H^s}+\tnorm{\ti{u}}{H^s}}^4+\big( X_s(u)+X_s(\ti{u})\big) \LR{\tnorm{u}{H^s}+\tnorm{\ti{u}}{H^s}}^3\tnorm{u-\ti{u}}{H^s}.}
Then, the desired difference estimates are shown by Lemma~\ref{lem:exp}.
\end{proof}


%
\smallskip
\subsection{Proof of Theorem~\ref{thm}}\label{sec:proof}

Let $1/2<s\le 1$ and $u,\ti{u}\in C_TH^s$ be two solutions (in the sense of distribution) of \eqref{mBO'} with a common initial datum at $t=0$.
Define the corresponding solutions $\om ,\ti{\om}\in C_T\ell ^2_s$ of \eqref{eq_om} by \eqref{utov} and \eqref{vtoom}.
By translation in time (if necessary), it suffices to prove $u(t)=\ti{u}(t)$ on $[0,T_*]$ for some $0<T_*\le T$.

Let us first justify computations of each normal form reduction step, following the argument in \cite[Section~5]{GKO13}.
We only see the first normal form reduction step to show the idea; then the general steps are justified in the same manner.
Let $W(t,n):=\F [e^{ikF[u(t)]}](n)$ for some $k\in \{ \pm 1,\pm 3\}$.
As seen above, $\p _t\om ,\p _tW\in C_T\ell ^2_{s-1}$ by the equations \eqref{eq_om}, \eqref{eq:W} and hence $\om (\cdot ,n),W(\cdot ,n)\in C^1([0,T])$ for each $n\in \Bo{Z}$, which justifies application of the product rule:
\eqq{&\p _t\bigg[ \frac{e^{it\Phi}}{i\Phi}\hat{m}_i\big( n,n_1,\dots ,n_5\big) \om _1(n_1)W_2(n_2)\om _3(n_3)W_4(n_4)\om _5(n_5)\bigg] \\
&=e^{it\Phi}\hat{m}_i\big( n,n_1,\dots ,n_5\big) \om _1(n_1)W_2(n_2)\om _3(n_3)W_4(n_4)\om _5(n_5)\\
&\quad +\frac{e^{it\Phi}}{i\Phi}\hat{m}_i\big( n,n_1,\dots ,n_5\big) (\p _t\om _1)(n_1)W_2(n_2)\om _3(n_3)W_4(n_4)\om _5(n_5)\\
&\quad +\cdots +\frac{e^{it\Phi}}{i\Phi}\hat{m}_i\big( n,n_1,\dots ,n_5\big) \om _1(n_1)W_2(n_2)\om _3(n_3)W_4(n_4)(\p _t\om _5)(n_5)}
for each $n,n_1,\dots ,n_5\in \Bo{Z}$.
Each term on the right-hand side is absolutely and uniformly in $t\in [0,T]$ summable over $\{ (n_1,\dots ,n_5):n=n_{12345},\,|\Phi |>M\}$ for each $n\in \Bo{Z}$, which is a consequence of (the proof of) the estimates on $\Sc{N}^{(J)},\Sc{N}^{(J+1)}$ in $C_T\ell ^2_{s-1}$ and those on $\Sc{N}^{(J)}_1,\Sc{R}^{(J)}$ in $C_T\ell ^2_s$ given in Proposition~\ref{prop}.
(This justifies the substitution of the equation \eqref{eq_om1} into $\Sc{N}^{(J)}_2$, since the summations after substitution are shown to be absolutely convergent.)
Moreover, the function on the left-hand side before differentiated in $t$ is also absolutely summable by the estimate on $\Sc{N}^{(J)}_0$ in $C_T\ell ^2_s$.
As a result, we can switch the summation and the differentiation in $t$ in the classical sense:
\eqq{
&\p _t\bigg[ \sum _{\mat{n=n_{12345}\\ |\Phi |>M}}\frac{e^{it\Phi}}{\Phi}\hat{m}_i\big( n,n_1,\dots ,n_5\big) \om _1(n_1)W_2(n_2)\om _3(n_3)W_4(n_4)\om _5(n_5)\bigg] \\
&=\sum _{\mat{n=n_{12345}\\  |\Phi |>M}}\p _t\bigg[ \frac{e^{it\Phi}}{\Phi}\hat{m}_i\big( n,n_1,\dots ,n_5\big) \om _1(n_1)W_2(n_2)\om _3(n_3)W_4(n_4)\om _5(n_5)\bigg]
}
for each $n\in \Bo{Z}$.
In such a way, the normal form reduction steps are justified for general solutions $u\in C_TH^s$ and $\om \in C_T\ell ^2_s$.

We next show that $\om$ satisfies the equation
\eq{eq_om5}{\om (t,n)&=\om (0,n)+\sum _{j=1}^\I \Big[ \Sc{N}^{(j)}_0[u,\om ](t,n)-\Sc{N}^{(j)}_0[u,\om ](0,n)\Big] \\[-5pt]
&\hx +\sum _{j=1}^\I \int _0^t\Big[ \Sc{N}^{(j)}_R[u,\om ](\tau ,n)+\Sc{N}^{(j)}_1[u,\om ](\tau ,n)\Big] \,d\tau +\sum _{j=0}^\I \int _0^t\Sc{R}^{(j)}[u,\om ] (\tau ,n)\,d\tau ,}
which is the limit of \eqref{eq_om4} as $J\to \I$.
By virtue of Proposition~\ref{prop} (i), the sequences (in $J$)
\eqs{\sum _{j=1}^J\Big[ \Sc{N}^{(j)}_0[u,\om ](t,n)-\Sc{N}^{(j)}_0[u,\om ](0,n)\Big] ,\\[-5pt]
\sum _{j=1}^J \int _0^t\Big[ \Sc{N}^{(j)}_R[u,\om ](\tau ,n)+\Sc{N}^{(j)}_1[u,\om ](\tau ,n)\Big] \,d\tau ,\qquad \sum _{j=0}^J\int _0^t\Sc{R}^{(j)}[u,\om ] (\tau ,n)\,d\tau}
are convergent in $C_T\ell ^2_s$ and 
\eqq{\int _0^t\Sc{N}^{(J+1)}[u,\om ] (\tau ,n)\,d\tau \to 0}
in $C_T\ell ^2_{s-1}$, if we choose $M>1$ such that $C_0M^{-\de}<1$.
Therefore, we can take the limit $J\to \I$ in $C_T\ell ^2_{s-1}$ and \eqref{eq_om5} is verified.
Note that \eqref{eq_om5} now holds in $C_T\ell ^2_s$ since the last term $\int _0^t\Sc{N}^{(J+1)}$ of less regularity has been eliminated in the limit equation.

The difference $\om -\ti{\om}$ thus satisfies the equation
\eqq{&\om (t,n)-\ti{\om}(t,n)\\
&=\sum _{j=1}^\I \Big[ \Big( \Sc{N}^{(j)}_0[u,\om ](t,n)-\Sc{N}^{(j)}_0[\ti{u},\ti{\om}](t,n)\Big) -\Big( \Sc{N}^{(j)}_0[u,\om ](0,n)-\Sc{N}^{(j)}_0[\ti{u},\ti{\om}](0,n)\Big) \Big] \\[-5pt]
&\hx +\sum _{j=1}^\I \int _0^t\Big[ \Big( \Sc{N}^{(j)}_R[u,\om ](\tau ,n)-\Sc{N}^{(j)}_R[\ti{u},\ti{\om}](\tau ,n)\Big) +\Big( \Sc{N}^{(j)}_1[u,\om ](\tau ,n)-\Sc{N}^{(j)}_1[\ti{u},\ti{\om}](\tau ,n)\Big) \Big] \,d\tau \\[-5pt]
&\hx +\sum _{j=0}^\I \int _0^t\Big[ \Sc{R}^{(j)}[u,\om ](\tau ,n)-\Sc{R}^{(j)}[\ti{u},\ti{\om}](\tau ,n)\Big] \,d\tau}
on $[0,T]$.
We apply Proposition~\ref{prop} (ii) and then invoke Lemma~\ref{lem:est-u} to get
\eqq{\norm{\om -\ti{\om}}{C_{T_*}\ell ^2_s}\le \Big\{ \big( 2+T_*M+T_*\ti{C}_0M^{1/2}\big) \frac{\ti{C}_0M^{-1/2}}{1-\ti{C}_0M^{-1/2}}+T_*\ti{C}_0\frac{1}{1-\ti{C}_0M^{-1/2}}\Big\} \\
\times (C+1)\norm{\om-\ti{\om}}{C_{T_*}\ell ^2_s}}
for $0<T_*\le T_0$ and $M>0$ such that $C_0M^{-\de}<1$ and $\ti{C}_0M^{-1/2}<1$, where $C>0$ and $T_0\in (0,T]$ are constants given in  Lemma~\ref{lem:est-u}.
Taking larger $M$ (if necessary) and then choosing $T_*$ sufficiently small, we obtain
\eqq{\tnorm{\om -\ti{\om}}{C_{T_*}\ell ^2_s}\le \frac{1}{2}\tnorm{\om -\ti{\om}}{C_{T_*}\ell ^2_s},}
and therefore $\tnorm{\om -\ti{\om}}{C_{T_*}\ell ^2_s}=0$.
Finally, Lemma~\ref{lem:est-u} concludes that $u(t)=\ti{u}(t)$ on $[0,T_*]$, which establishes Theorem~\ref{thm}.\hfill $\Box$


\bigskip
\section{Proof of fundamental quintilinear estimates}\label{sec:proof-multi}

\smallskip
All we have to do is to prove the fundamental quintilinear estimates given in Proposition~\ref{prop:fundamental}, which is the goal of this section.
First, we further reduce Proposition~\ref{prop:fundamental} to the following lemmas:
\begin{lem}\label{lem:1}
Let $s>1/2$, then there exists $\de >0$ such that we have 
\begin{align}
\begin{split}
&\norm{\sum _{\mat{n=n_{12345},(n_l)\in \Sc{A}_1\\ n>0,n_{45}<0}}n_{45}\om _1(n_1)W_2(n_2)\om _3(n_3)W_4(n_4)\om _5(n_5)}{\ell ^2_s} \\[-10pt]
&\hspace{90pt} \lec \norm{\om _1}{\ell ^2_s}\norm{\om _3}{\ell ^2_s}\norm{\om _5}{\ell ^2_s}\Big( \norm{W_2}{\ell ^2_{s+1}}\norm{W_4}{\ell ^2_{s}}+\norm{W_2}{\ell ^2_{s}}\norm{W_4}{\ell ^2_{s+1}}\Big) ,
\end{split}\label{est:5linear-0}\\[10pt]
\begin{split}
&\norm{\sum _{\mat{n=n_{12345},(n_l)\not\in \Sc{A}_1\\ n>0,n_{45}<0}}\frac{n_{45}}{\LR{\Phi}^{1/2}}\om _1(n_1)W_2(n_2)\om _3(n_3)W_4(n_4)\om _5(n_5)}{\ell ^2_s} \\[-10pt]
&\hspace{200pt} \lec \norm{\om _1}{\ell ^2_s}\norm{W_2}{\ell ^2_s}\norm{\om _3}{\ell ^2_s}\norm{W_4}{\ell ^2_s}\norm{\om _5}{\ell ^2_s},
\end{split}\label{est:5linear-1}\\[10pt]
\begin{split}
&\norm{\sum _{\mat{n=n_{12345},(n_l)\not\in \Sc{A}_1\\ n>0,n_{45}<0}}\frac{n_{45}}{\LR{\Phi}^{1-\de}}\frac{\LR{n_{\max}}}{\LR{n}}\om _1(n_1)W_2(n_2)\om _3(n_3)W_4(n_4)\om _5(n_5)}{\ell ^2_s} \\[-10pt]
&\hspace{200pt} \lec \norm{\om _1}{\ell ^2_s}\norm{W_2}{\ell ^2_s}\norm{\om _3}{\ell ^2_s}\norm{W_4}{\ell ^2_s}\norm{\om _5}{\ell ^2_s}
\end{split}\label{est:5linear-2}
\end{align}
for any real-valued non-negative functions $\om _1,\om _3,\om _5$ and $W_2,W_4$, where $n_{\max}:=\max _{1\le j\le 5}|n_j|$.
\end{lem}

\begin{lem}\label{lem:2}
Let $s>1/2$, then there exists $\de >0$ such that we have 
\begin{align}
\begin{split}
&\norm{\sum _{\mat{n=n_{12345},(n_l)\in \Sc{A}_2\\ n_{23}>0,n_{45}>0}}\frac{n_{23}n_{45}}{n_{2345}}\om _1(n_1)W_2(n_2)\om _3(n_3)W_4(n_4)\om _5(n_5)}{\ell ^2_s} \\[-10pt]
&\hspace{90pt} \lec \norm{\om _1}{\ell ^2_s}\norm{\om _3}{\ell ^2_s}\norm{\om _5}{\ell ^2_s}\Big( \norm{W_2}{\ell ^2_{s+1}}\norm{W_4}{\ell ^2_{s}}+\norm{W_2}{\ell ^2_{s}}\norm{W_4}{\ell ^2_{s+1}}\Big) ,
\end{split}\label{est:6linear-0}\\[10pt]
\begin{split}
&\norm{\sum _{\mat{n=n_{12345},(n_l)\not\in \Sc{A}_2\\ n_{23}>0,n_{45}>0}}\frac{n_{23}n_{45}}{\LR{\Phi}^{1/2}n_{2345}}\om _1(n_1)W_2(n_2)\om _3(n_3)W_4(n_4)\om _5(n_5)}{\ell ^2_s} \\[-10pt]
&\hspace{200pt} \lec \norm{\om _1}{\ell ^2_s}\norm{W_2}{\ell ^2_s}\norm{\om _3}{\ell ^2_s}\norm{W_4}{\ell ^2_s}\norm{\om _5}{\ell ^2_s},
\end{split}\label{est:6linear-1}\\[10pt]
\begin{split}
&\norm{\sum _{\mat{n=n_{12345},(n_l)\not\in \Sc{A}_2\\ n_{23}>0,n_{45}>0}}\frac{n_{23}n_{45}}{\LR{\Phi}^{1-\de}n_{2345}}\frac{\LR{n_{\max}}}{\LR{n}}\om _1(n_1)W_2(n_2)\om _3(n_3)W_4(n_4)\om _5(n_5)}{\ell ^2_s} \\[-10pt]
&\hspace{200pt} \lec \norm{\om _1}{\ell ^2_s}\norm{W_2}{\ell ^2_s}\norm{\om _3}{\ell ^2_s}\norm{W_4}{\ell ^2_s}\norm{\om _5}{\ell ^2_s}
\end{split}\label{est:6linear-2}
\end{align}
for any real-valued non-negative functions $\om _1,\om _3,\om _5$, $W_2,W_4$.
\end{lem}

These lemmas yield Proposition~\ref{prop:fundamental} as follows.

\begin{proof}[Proof of Proposition~\ref{prop:fundamental}]
\eqref{est:matome-0}:
We divide $|\hat{\hat{m}}|$ as $|m|\chf{(n_l)\in \Sc{A}}+|m|\chf{(n_l)\not \in \Sc{A}}\chf{|\Phi |\le |n_2|^2\vee |n_4|^2}$, where $m$ is one of $\shugo{m_i\,;\,1\le i\le 7}$ and $\Sc{A}$ is $\Sc{A}_1$ ($1\le i\le 5$) or $\Sc{A}_2$ ($i=6,7$).
For the first term, \eqref{est:5linear-0} and \eqref{est:6linear-0} yield the claim.
For the second term, we use \eqref{est:5linear-1} and \eqref{est:6linear-1} after multiplying by $\big( \LR{n_2}+\LR{n_4}\big) /\LR{\Phi}^{1/2}\ge 1$.

\eqref{est:matome-1}:
This immediately follows from \eqref{est:5linear-1} and \eqref{est:6linear-1}.

\eqref{est:matome-2}:
We multiply the quintilinear form by $\LR{\Phi}^{1/2}/\big( \LR{n_2}+\LR{n_4}\big)\gec 1$ and apply \eqref{est:5linear-1}, \eqref{est:6linear-1}.

\eqref{est:matome-3}:
This immediately follows from \eqref{est:5linear-2} and \eqref{est:6linear-2}.
\end{proof}

To show Lemmas~\ref{lem:1} and \ref{lem:2}, we prepare the following combinatorial tool.
\begin{lem}\label{lem:number}
Let $\e >0$, then we have
\begin{align}
\sup _{n_1^*,n_2^*,\mu ^*\in \Bo{Z}}\Big[ &\# \Shugo{(n_1,n_2)\in \Bo{Z}^2}{3(n_1-n_1^*)^2+(n_2-n_2^*)^2=\mu ^*,\,|n_1|+|n_2|\le R} \label{number1}\\[-5pt]
&+\# \Shugo{(n_1,n_2)\in \Bo{Z}^2}{(n_1-n_1^*)(n_2-n_2^*)=\mu ^*\neq 0,\,|n_1|+|n_2|\le R}\Big] \label{number2}\\
\lec _\e R^\e &\notag
\end{align}
for any $R>1$.
\end{lem}
\begin{proof}
Estimate for \eqref{number1} follows via well-known divisor counting argument and Jarn\'ik's geometric observation; see, e.g., \cite[Proposition~2.36]{B93-1} and \cite[Lemma~1.5]{B07}.
Estimate for \eqref{number2} is immediate if $|\mu ^*|\lec R^6$.
When $|\mu ^*|\gg R^6$, we may assume $|n_1-n_1^*|\ge |n_2-n_2^*|$, which implies $|n_1-n_1^*|\ge |\mu ^*|^{1/2}\gg R^3$.
It turns out that the number of such points $(n_1,n_2)$ is at most two; this kind of argument is also well-known, see \mbox{e.g.} \cite[Lemma~6.1]{CKSTT04}.
In fact, suppose that there exist three distinct points $(n_{1,j},n_{2,j})$, $j=1,2,3$ satisfying
\eqq{(n_{1,j}-n_1^*)(n_{2,j}-n_2^*)=\mu ^*\neq 0,\quad |n_{1,j}|+|n_{2,j}|\le R,\quad p_j:=|n_{1,j}-n_1^*|\ge |\mu ^*|^{1/2}\gg R^3,}
then for each $j$, $p_j$ is a divisor of $\mu ^*$ and confined to an interval of length $2R$.
Moreover, $p_j$, $j=1,2,3$ are mutually different. 
Then, using the identity
\eqq{\mathrm{lcm}(p_1,p_2,p_3) \gcd (p_1,p_2)\gcd (p_1,p_3)\gcd (p_2,p_3)=p_1p_2p_3\cdot \gcd (p_1,p_2,p_3),}
we deduce a contradiction
\[ 8R^3|\mu ^*|\ge |\mu ^*|^{3/2}\gg R^3|\mu ^*|.\qedhere \]
\end{proof}

Now, we are in a position to prove Lemmas~\ref{lem:1} and \ref{lem:2}.
This is done by a thorough case-by-case analysis, which reveals how we come to the elaborate definition of the sets $\Sc{A}_1,\Sc{A}_2$. 

\begin{proof}[Proof of Lemma~\ref{lem:1}]

\underline{[0] Proof of \eqref{est:5linear-0}}.

[0.1] $n_{15}n_{35}=0$.
We consider the case $n_1+n_5=0$ without loss of generality.
Noticing $s>1/2$ and using Young's inequality, we have
\eqq{\text{LHS of \eqref{est:5linear-0}}&\le \norm{\sum _{n=n_{12345}}\om _1W_2\om _3|n_4|W_4\om _5}{\ell ^2_s}+\norm{\sum _{n=n_{234}}W_2\om _3W_4}{\ell ^2_s}\sum _{n_5}|n_5|\om _1(-n_5)\om _5(n_5)\\
&\lec \norm{\om _1}{\ell ^2_s}\norm{W_2}{\ell ^2_s}\norm{\om _3}{\ell ^2_s}\norm{W_4}{\ell ^2_{s+1}}\norm{\om _5}{\ell ^2_s}+\norm{W_2}{\ell ^2_{s}}\norm{\om _3}{\ell ^2_s}\norm{W_4}{\ell ^2_{s}}\norm{\om _1}{\ell ^2_{1/2}}\norm{\om _5}{\ell ^2_{1/2}}.}

[0.2] $|n_2|\vee |n_4|\ge \eta ^2(|n_5|\wedge |n|)$.
If $|n_2|\vee |n_4|\ge \eta ^2|n_5|$, we have $|n_{45}|\lec |n_2|+|n_4|$, and the desired estimate is easily verified by Young's inequality.
If $\eta ^2|n_5|>|n_2|\vee |n_4|\ge \eta ^2|n|$, then
\eqq{\text{LHS of \eqref{est:5linear-0}}\lec \norm{\sum _{n=n_{12345}}\big( \LR{n_2}+\LR{n_4}\big) \om _1W_2\om _3W_4\LR{n_5}\om _5}{\ell ^2_{s-1}},}
and the claim follows from the Sobolev estimate \eqref{prod-est1}.

[0.3] $|n_2|\vee |n_4|<\eta ^2|n_5|$, $|n_5|<\eta (|n_1|\wedge |n_3|)$, and $|n_{24}|\ge \eta |n_{13}|$.
Replacing $|n_{45}|$ with $|n_5|$, we may assume $|n_2|\ge |n_4|$ by symmetry.
Divide $n_2,n_4,n_5$ into dyadic parts $\shugo{\LR{n_l}\sim N_l}$ and restrict $n_1$ and $n_3$ into intervals $Q_1,Q_3$ with length $N_2$.
Then, we have $\LR{n}^s|n_{45}|\lec \LR{n_1}^{s}\LR{n_3}^{s}\LR{n_5}^{1-s}$ and
\eqq{&\norm{\sum _{n=n_{12345}}|n_{45}|\om _1W_2\om _3W_4\om _5}{\ell ^2_s}\\
&\lec N_2^{1/2}\norm{\chi _{Q_1}\om _1}{\ell ^2_s}\norm{\chi _{\sim N_2}W_2}{\ell ^2_{1/2}}N_2^{1/2}\norm{\chi _{Q_3}\om _3}{\ell ^2_s}\norm{\chi _{\sim N_4}W_4}{\ell ^2_{1/2}}\norm{\chi _{\sim N_5}\om _5}{\ell ^2_{1-s}}\\
&\lec N_2^{1/2-s}N_4^{1/2-s}N_5^{1-2s}\norm{\chi _{Q_1}\om _1}{\ell ^2_s}\norm{\chi _{\sim N_2}W_2}{\ell ^2_{s+1}}\norm{\chi _{Q_3}\om _3}{\ell ^2_s}\norm{\chi _{\sim N_4}W_4}{\ell ^2_{s}}\norm{\chi _{\sim N_5}\om _5}{\ell ^2_{s}}.
}
By the almost orthogonality $|n_{13}|\lec N_2$ the summation over $Q_1,Q_3$ is performed via the Cauchy-Schwarz inequality.
We can sum up over $N_2,N_4,N_5$ since $s>1/2$, and the estimate follows.

[0.4] $|n_2|\vee |n_4|<\eta ^2|n_5|$, $|n_3|<\eta (|n_1|\wedge |n_5|)$, $|n_5|\le 2|n_1|$, and $|n_{24}|\ge \eta |n_{15}|$.
Similarly to the preceding case, divide $n_2,n_3,n_4$ into dyadic parts and $n_1,n_5$ into intervals of length $N_2$, assuming $N_2\ge N_4$.
Since $|n|\le |n_{24}|+|n_{15}|+|n_3|\lec |n_2|+|n_3|$ and thus $\LR{n}^s|n_{45}|\lec (N_2+N_3)^{1-s}\LR{n_1}^{s}\LR{n_5}^s$, we have a similar estimate as [0.3].

\smallskip
\underline{[1] Proof of \eqref{est:5linear-1}}.
Recall that for $(n_1,\dots ,n_5)\not\in \Sc{A}_1$ we have $n_{15}n_{35}\neq 0$ and $|n_2|\vee |n_4|<\eta ^2(|n_5|\wedge |n|)$, which implies $|n_{45}|\lec |n_5|$.
By symmetry we may assume $|n_1|\ge |n_3|$.
Since $n>0$ and $n_{45}<0$, it must hold that $n_{13},n_1>0$ and $n_5<0$, and particularly $n_{13}n_{15}n_{35}\neq 0$.
We also note that $|n_1|\sim n_{\max}$ in this case.

We see that the following identities hold:
\begin{equation}
\Phi =n^2-n_1^2-n_3|n_3|+n_5^2=\begin{cases}
2n_{15}n_{35}+n^2-p^2 &\text{if $n_3\ge 0$},\\
-2n_{13}n_{15}+n^2+p^2&\text{if $n_3<0$},\\
\end{cases}\label{Phi-5}
\end{equation}
where $p:=n_{135}=n-n_{24}$.
Then, the following fact is verified by Lemma~\ref{lem:number}:
\eq{fact1}{&\text{If two of $n_1,n_3,n_5$ are restricted into intervals of length $R>1$, then the number of}\\[-5pt]
&\text{possible choices of $(n_1,n_3,n_5)$ for fixed $n=n_{12345}$, $n_2$, $n_4$, $\Phi$, is at most $O_\e (R^\e )$.}}

By the Cauchy-Schwarz inequality, \eqref{est:5linear-1} is reduced to showing
\eq{claim5-1}{\sum _{n=n_{12345}}\frac{|n_5|^2\LR{n}^{2s}}{\LR{\Phi}\LR{n_1}^{2s}\LR{n_2}^{2s}\LR{n_3}^{2s}\LR{n_4}^{2s}\LR{n_5}^{2s}}\lec 1}
for $n\in \Bo{Z}$ uniformly.

[1.1] $|n_3|\ge \eta |n_5|$.
In this case we have $|n_5|^2\LR{n}^{2s}\lec \LR{n_5}^{2s}\LR{n_3}^{2(1-s)}\LR{n_1}^{2s}$.
Decomposing the summation in $n_3$ into dyadic intervals and invoking \eqref{fact1}, we have
\eqq{&\text{LHS of \eqref{claim5-1}}\\[-5pt]
&\lec \sum _{N_3\ge 1}\frac{1}{N_3^{2(2s-1)}}\sum _{n_2,n_4}\frac{1}{\LR{n_2}^{2s}\LR{n_4}^{2s}}\sum _{\mu \in \Bo{Z}}\frac{\# \Shugo{(n_1,n_3,n_5)}{n=n_{12345},\,\Phi =\mu,\, |n_3|+|n_5|\lec N_3}}{\LR{\mu}}\\
&\lec \sum _{N_3\ge 1}\frac{N_3^\e \log N_3}{N_3^{2(2s-1)}}\lec 1}
whenever $0<\e<2(2s-1)$, where we have used the fact that $\Phi$ can take at most $O(N_3^2)$ values for fixed $n,n_2,n_4$.

[1.2] $\eta |n|\le |n_3|<\eta |n_5|$.
It holds $|n_5|^2\LR{n}^{2s}\lec \LR{n_5}^{2s}\LR{n_1}^{2(1-s)}\LR{n_3}^{2s}$.
By a similar argument as [1.1], we see that
\eqq{&\text{LHS of \eqref{claim5-1}}\\[-5pt]
&\lec \sum _{N_1\ge 1}\frac{1}{N_1^{2(2s-1)}}\sum _{n_2,n_4}\frac{1}{\LR{n_2}^{2s}\LR{n_4}^{2s}}\sum _{\mu \in \Bo{Z}}\frac{\# \Shugo{(n_1,n_3,n_5)}{n=n_{12345},\,\Phi =\mu,\, |n_1|+|n_3|\lec N_1}}{\LR{\mu}}\\
&\lec \sum _{N_1\ge 1}\frac{N_1^\e \log N_1}{N_1^{2(s+1)}}\lec 1}
if $0<\e<2(2s-1)$.

[1.3] $|n_3|<\eta (|n_5|\wedge |n|)$.
In this case it holds that $n_{15}=n-n_{234}\ge (1-3\eta )n>0$.
We observe that
\eqq{\Phi &=(n_{15}+n_{234})^2+n_5^2-(n_{15}-n_5)^2-n_3|n_3|\le 2n_{15}(n_5+n_{234})+n_{234}^2+n_3^2\\
&\le -2(1-3\eta )^2n|n_5|+9\eta ^2n|n_5|+\eta ^2n|n_5|\le -|n||n_5|.}
Therefore, we have $|n_5|^2\LR{n}^{2s}\lec \LR{\Phi}\LR{n_1}^{2s}$ and
\eqq{\text{LHS of \eqref{claim5-1}}&\lec \sum _{n_2,n_3,n_4,n_5}\frac{1}{\LR{n_2}^{2s}\LR{n_3}^{2s}\LR{n_4}^{2s}\LR{n_5}^{2s}}\lec 1.}

\smallskip
\underline{[2] Proof of \eqref{est:5linear-2}}.

[2.1] $|n|\ge \eta ^2n_{max}$.
In this case the estimate is reduced to \eqref{est:5linear-1}.

[2.2] $|\Phi|\ge \eta ^3n_{\max}^2$.
We have
\eqq{\frac{\LR{n_{\max}}}{\LR{\Phi}^{1/2-\de}\LR{n}}\lec \frac{\LR{n_{\max}}^{2\de}}{\LR{n}}\lec \frac{\LR{n_1}^{2\de}\LR{n_2}^{2\de}\LR{n_3}^{2\de}\LR{n_4}^{2\de}\LR{n_5}^{2\de}}{\LR{n}^{2\de}}}
if $0<\de \le 1/2$, hence this case is also reduced to [1] with $s$ replaced by $s-2\de$ once we choose $\de>0$ such that $s-2\de >1/2$.

[2.3] $|n|<\eta ^2n_{max}$ and $|\Phi|<\eta ^3n_{\max}^2$.
Recall that we are assuming $n_{15}n_{35}\neq 0$ and $|n_2|\vee |n_4|<\eta ^2(|n_5|\wedge |n|)$, which implies $|n_{45}|\lec |n_5|$.
We may assume $|n_1|\ge |n_3|$, and thus it holds that $n,n_1,n_{13}>0$, $n_5,n_{45}<0$, $|n_1|\sim n_{\max}$.

[2.3.1] $|n_5|<\eta |n_3|$.
In this case, $|n|<\eta ^2n_{max}~(=\eta ^2|n_1|)$ implies $n_3<0$ and $|n_1|\sim |n_3|$.
Since $(n_j)\not\in \Sc{A}_1$, we also have $|n_{24}|<\eta |n_{13}|$.
Moreover, we have $n<2n_{13}$, since otherwise $0<n/2\le n-n_{13}=n_5+n_{24}\le n_5+2\eta ^2|n_5|<0$, which is a contradiction.
Hence, \eqref{Phi-5} implies that
\eqq{\Phi &=-2n_{13}n_{15}+n^2+(n-n_{24})^2\\
&\le -2n_{13}(1-\eta )n_1+2n_{13}\cdot \eta ^2n_1+(2n_{13}+\eta n_{13})(\eta ^2n_1+2\eta ^3n_1)\le -|n_{13}|n_{\max},}
which then yields that
\eqq{\frac{|n_5|^2\LR{n_{\max}}^2}{\LR{\Phi}^{2(1-\de )}}\lec \frac{\LR{n_1}^{2(1-s+\de )}\LR{n_5}^{2s}}{|n_{13}|^{2(1-\de )}}.}
If we choose $\de >0$ such that $1-s+\de \le s$, we have
\eqq{&\sum _{n=n_{12345}}\frac{|n_5|^2\LR{n_{\max}}^2}{\LR{\Phi}^{2(1-\de )}\LR{n}^{2(1-s)}\LR{n_1}^{2s}\LR{n_2}^{2s}\LR{n_3}^{2s}\LR{n_4}^{2s}\LR{n_5}^{2s}}\\
&\hx \lec \sum _{n_1,n_2,n_3,n_4}\frac{1}{|n_{13}|^{2(1-\de )}\LR{n_2}^{2s}\LR{n_3}^{2s}\LR{n_4}^{2s}}\lec 1.}
By the Cauchy-Schwarz inequality we have the claim.

[2.3.2] $|n_3|<\eta (|n_1|\wedge |n_5|)$.
Since $|n|<\eta ^2n_{max}$, we have $|n_1|, |n_5|\ge n_{\max}/2$, and then the assumption $(n_j)\not\in \Sc{A}_1$ implies $|n_{24}|<\eta |n_{15}|$.
We see that $|\Phi |\gec |n_{15}|n_{\max}$ in this case.
Indeed, from \eqref{Phi-5}, if $n_3\ge 0$ we have
\eqq{|\Phi |&=|2n_{15}n_{35}+n_{24}(2n-n_{24})|\\
&\ge 2(1-\eta )|n_{15}||n_5|-\eta |n_{15}|(4\eta ^2|n_5|+2\eta ^2|n_5|)\ge |n_{15}||n_5|\ge |n_{15}|n_{\max}/2,}
whereas if $n_3<0$ we have $|n|<|n-n_3|\le |n_{15}|+|n_{24}|\le (1+\eta )|n_{15}|$ and 
\eqq{|\Phi |&=|-2n_{13}n_{15}+n^2+(n-n_{24})^2|\\
&\ge 2(1-\eta )|n_1||n_{15}|-2\eta ^2(1+\eta )|n_1||n_{15}|-6\eta ^2(1+2\eta )|n_1||n_{15}|\ge |n_1||n_{15}|\ge |n_{15}|n_{\max}/2.}
To show the desired estimate it suffices to prove for dyadic $N_*\ge 1$ that
\eqq{&\norm{\sum _{\mat{n=n_{12345}\\ |n_{15}|\sim N_*}}\frac{|n_5|\LR{n_{\max}}}{|n_{15}|^{1-\de}n_{\max}^{1-\de}\LR{n}}\om _1(n_1)W_2(n_2)\om _3(n_3)W_4(n_4)\om _5(n_5)}{\ell ^2_s}\\
&\hx \lec N_*^{-\e}\norm{\om _1}{\ell ^2_{s}}\norm{W_2}{\ell ^2_s}\norm{\om _3}{\ell ^2_s}\norm{W_4}{\ell ^2_{s}}\norm{\om _5}{\ell ^2_s}}
for some $\e >0$ (recall that $n_{15}\neq 0$).
For each $N_*$ we may restrict the values of $n_1$ and $n_5$ onto intervals of length $N_*$ by the almost orthogonality.
Then, the left-hand side of the above is evaluated by
\eqq{N_*^{-(1-\de )}N_*^{1/2}\norm{\om _1}{\ell ^2_{1-s+\de}}\norm{W_2}{\ell ^2_s}\norm{\om _3}{\ell ^2_s}\norm{W_4}{\ell ^2_s}\norm{\om _5}{\ell ^2_s}.}
We obtain the claim by choosing $\de >0$ so that $\de <1/2$ and $1-s+\de \le s$.

[2.3.3] $\eta ^{-1}|n_5|\ge |n_3|\ge \eta (|n_1|\wedge |n_5|)$.
We observe that $|n|<\eta ^2n_{\max}$ and $|n_{24}|<2\eta ^2|n_5|$ imply $|n_{135}|<3\eta ^2n_{\max}$.
Since $n_{\max}=|n_1|\vee |n_5|$, it then holds that $|n_1|,|n_3|,|n_5|\ge (\eta /4)n_{\max}$ in this case.
Here, we see that $|n_{13}|,|n_{15}|,|n_{35}|\ge (\eta /8)n_{\max}$.
In fact, if $|n_{13}|<(\eta /8)n_{\max}$ then $|n_5|\le |n_{135}|+|n_{13}|\le (3\eta ^2+\eta /8)n_{\max}$, which contradicts $|n_{5}|\ge (\eta /4)n_{\max}$, and similarly for the others.
Now, \eqref{Phi-5} shows that
\eqq{|\Phi |\ge (\eta ^2/32)n_{\max}^2-10\eta ^4n_{\max}^2\ge (\eta ^2/64)n_{\max}^2,}
which contradicts the assumption $|\Phi |<\eta ^3n_{\max}^2$.
Therefore, this case does not occur.
\end{proof}

\begin{proof}[Proof of Lemma~\ref{lem:2}]
We first observe that $n_{23},n_{45}>0$ implies $0<\frac{n_{23}n_{45}}{n_{2345}}\sim n_{23}\wedge n_{45}$.

\underline{[0] Proof of \eqref{est:6linear-0}}.

[0.1] $n_{13}n_{15}=0$.
Since $\frac{n_{23}n_{45}}{n_{2345}}\lec (|n_2|+|n_3|)\wedge (|n_4|+|n_5|)$, a similar argument as Case [0.1] of  the proof of Lemma~\ref{lem:1} suffices.

[0.2] $|n_2|\vee |n_4|\ge \eta (|n_3|\wedge |n_5|)$.
Since $\frac{n_{23}n_{45}}{n_{2345}}\lec |n_2|\vee |n_4|$, this case is handled easily.

\smallskip
\underline{[1] Proof of \eqref{est:6linear-1}}.
Recall that for $(n_1,\dots ,n_5)\not\in \Sc{A}_2$ we have $n_{13}n_{15}\neq 0$ and $|n_2|\vee |n_4|<\eta (|n_3|\wedge |n_5|)$.
By symmetry we may assume $|n_3|\ge |n_5|$, which implies $\frac{n_{23}n_{45}}{n_{2345}}\lec |n_5|$ and $n_{\max}=|n_1|\vee |n_3|$.
Moreover, it must hold that $n_3\ge n_5>0$, and we have the following identities:
\begin{equation}
\Phi =\begin{cases}
-\frac{1}{6}\big\{ 3(2n_1+n_3-p)^2+(3n_3-p)^2\big\} +n|n|-\frac{1}{3}p^2 &\text{if $n_1\ge 0$},\\
2n_{13}n_{15}+n|n|-p^2 &\text{if $n_1<0$},\\
\end{cases}\label{Phi-6}
\end{equation}
where $p:=n_{135}=n-n_{24}$.
Then, the following fact is verified by Lemma~\ref{lem:number}:
\eq{fact2}{&\text{If two of $n_1,n_3,n_5$ are restricted into intervals of length $R>1$, then the number of}\\[-5pt]
&\text{possible choices of $(n_1,n_3,n_5)$ for fixed $n=n_{12345}$, $n_2$, $n_4$, and $\Phi$, is at most $O_\e (R^\e )$.}}
By the Cauchy-Schwarz inequality, \eqref{est:6linear-1} is again reduced to showing \eqref{claim5-1} for $n\in \Bo{Z}$ uniformly.

[1.1] $|n_1|\ge \eta |n_5|$.
Since $|n|\lec n_{\max}=|n_1|\vee |n_3|$, it holds $|n_5|^2\LR{n}^{2s}\lec \LR{n_3}^{2(1-s)}\LR{n_5}^{2s}\LR{n_1}^{2s}$ or $|n_5|^2\LR{n}^{2s}\lec \LR{n_1}^{2(1-s)}\LR{n_5}^{2s}\LR{n_3}^{2s}$.
In the former case, we decompose $n_3$ dyadically and using the fact \eqref{fact2} to obtain
\eqq{&\text{LHS of \eqref{claim5-1}}\\[-5pt]
&\lec \sum _{N_3\ge 1}\frac{1}{N_3^{2(2s-1)}}\sum _{n_2,n_4}\frac{1}{\LR{n_2}^{2s}\LR{n_4}^{2s}}\sum _{\mu \in \Bo{Z}}\frac{\# \Shugo{(n_1,n_3,n_5)}{n=n_{12345},\,\Phi =\mu ,\,|n_3|+|n_5|\lec N_3}}{\LR{\mu}}\\
&\lec \sum _{N_3\ge 1}\frac{N_3^\e \log N_3}{N_3^{2(2s-1)}}\lec 1,}
whenever $0<\e <2(2s-1)$, where we have used the fact that $\Phi$ can take at most $O(N_3^2)$ different values.
In the latter case, it suffices to make a similar argument but decomposing $n_1$ instead of $n_3$.

[1.2] $|n_1|<\eta |n_5|$.
Since $|n_{124}|<3\eta |n_5|$, we see that $n=n_3+n_5+n_{124}\ge n_3+(1-3\eta )n_5>0$, 
\eqq{|\Phi |&=\big| n^2-n_1|n_1|-n_3^2-n_5^2\big| =\big| 2n_3n_5+2n_{35}n_{124}+n_{124}^2-n_1|n_1|\big| \\
&\ge 2|n_3||n_5|-12\eta |n_3||n_5|-9\eta ^2|n_3||n_5|-\eta ^2|n_3||n_5|\ge |n_3||n_5|.}
This implies $|n_5|^2\LR{n}^{2s}\lec \LR{\Phi}\LR{n_3}^{2s}$, which immediately yields \eqref{claim5-1}.

\smallskip
\underline{[2] Proof of \eqref{est:6linear-2}}.
We assume $|n_3|\ge |n_5|$ again.
Recall that $n_{13}n_{15}\neq 0$, $|n_2|\vee |n_4|<\eta |n_5|$, and $n_3\ge n_5>0$.
By the Cauchy-Schwarz inequality it suffices to show that
\eq{claim6-2}{\sum _{n=n_{12345}}\frac{|n_5|^2\LR{n_{\max}}^2}{\LR{\Phi}^{2(1-\de )}\LR{n}^{2(1-s)}\LR{n_1}^{2s}\LR{n_2}^{2s}\LR{n_3}^{2s}\LR{n_4}^{2s}\LR{n_5}^{2s}}\lec 1}
uniformly for $n\in \Bo{Z}$.

[2.1] $|n|\ge \eta n_{\max}$.
This case is reduced to \eqref{est:6linear-1}.

[2.2] $|n|<\eta n_{\max}$.
Since $|n_{135}|\le |n|+|n_{24}|<3\eta n_{\max}$ and $n_3\ge n_5>0$, it must hold that $n_1<0$ and $|n_1|\sim |n_3|\sim n_{\max}$.
Consider two subcases separately.

[2.2.1] $|n|\ge \eta |n_5|$.
In this case we have $|n_5|^2\LR{n_{\max}}^2\lec \LR{n}^{2(1-s)}\LR{n_5}^{2s}\LR{n_1}^{2s}\LR{n_3}^{2(1-s)}$, and
\eqq{\text{LHS of \eqref{claim6-2}}\lec \sum _{n=n_{12345}}\frac{1}{\LR{\Phi}^{2(1-\de )}\LR{n_2}^{2s}\LR{n_4}^{2s}\LR{n_3}^{2(2s-1)}}.}
Hence, the claim follows from the fact \eqref{fact2} in the same way as Case [1.1].

[2.2.2] $|n|<\eta |n_5|$.
We have $|n_{13}|\ge |n_5|-|n|-|n_{24}|\ge (1-3\eta )|n_5|$ and $|n_{15}|\ge |n_3|-|n|-|n_{24}|\ge (1-3\eta )|n_3|$, and thus by \eqref{Phi-6}
\eqq{|\Phi |=\big| 2n_{13}n_{15}+n|n|-(n-n_{24})^2\big| \ge 2(1-3\eta )^2|n_3||n_5|-\eta ^2|n_3||n_5|-9\eta ^2|n_3||n_5|\ge |n_3||n_5|.}
This implies $|n_5|^2\LR{n_{\max}}^2\lec \LR{\Phi}^{2(1-\de )}\LR{n_5}^{2\de}\LR{n_3}^{2\de}$, which yields \eqref{claim6-2} if $2(s-\de )>1$.
\end{proof}


\bigskip
\section*{Acknowledgments}
The result in this article was obtained during the author's stay at the University of Chicago in 2013--2014.
He is deeply grateful to Professor Carlos E. Kenig for fruitful discussions and also to all members of the Department of Mathematics for their heartwarming hospitality.
The author has been partially supported by JSPS KAKENHI Grant-in-Aid for Young Researchers (B) \#%
JP24740086 and \#%
JP16K17626.



\bigskip

\end{document}